\documentclass{emsprocart}
\usepackage[all]{xy}
\usepackage{makeidx}

\contact[cautis@usc.edu]{Sabin Cautis, Department of Mathematics\\ University of Southern California \\ Los Angeles, CA}

\newtheorem{theorem}{Theorem}[section]
\newtheorem{corollary}[theorem]{Corollary}
\newtheorem{lemma}[theorem]{Lemma}
\newtheorem{proposition}[theorem]{Proposition}
\newtheorem{conjecture}[theorem]{Conjecture}

\theoremstyle{definition}

\newtheorem{remark}[theorem]{Remark}

\renewcommand{\theenumi}{\roman{enumi}}

\def\Dmod{\mbox{Dmod}}
\def\Gr{{{Gr}}}
\def\sK{{\mathcal{K}}}
\def\sM{{\mathcal{M}}}
\def\sN{{\mathcal{N}}}
\def\sIL{{\mathcal{IL}}}
\def\ip{\mbox{ip}}
\def\IZ{\mbox{IZ}}
\def\IB{\mbox{IB}}
\def\Bl{\mbox{Bl}}
\def\tr{\mbox{tr}}
\def\sA{\mathcal{A}}
\def\tsA{\tilde{\mathcal{A}}}
\def\tsE{\tilde{\mathcal{E}}}

\def\Bl{\mbox{Bl}}
\def\tZ{\tilde{Z}}
\def\sL{\mathcal{L}}
\def\wt{\widetilde}
\def\tp{\tilde{p}}
\def\tB{\tilde{B}}
\def\P{\mathbb{P}}

\def\A{{\mathbb{A}}}
\def\id{{\mathrm{id}}}
\def\l{{\lambda}}

\def\sl{\mathfrak{sl}}
\def\so{\mathfrak{so}}
\def\g{\mathfrak{g}}
\def\fh{\mathfrak{h}}
\def\O{{\mathcal O}}

\def\sE{{\mathcal{E}}}

\def\tsE{{\tilde{\mathcal{E}}}}
\def\sF{{\mathcal{F}}}
\def\sL{{\mathcal{L}}}

\def\sH{{\mathcal{H}}}
\def\sT{{\mathcal{T}}}
\def\sP{{\mathcal{P}}}
\def\sQ{{\mathcal{Q}}}
\def\bG{\mathbb{G}}
\def\pt{\mathrm{pt}}
\def\tY{\tilde{Y}}
\newcommand{\IG}{\mathbb{IG}}
\newcommand{\Z}{\mathbb{Z}}
\newcommand{\C}{\mathbb{C}}
\newcommand{\Cone}{\operatorname{Cone}}

\newcommand{\la}{\langle}
\newcommand{\ra}{\rangle}
\newcommand{\Ext}{\operatorname{Ext}}
\newcommand{\Hom}{\operatorname{Hom}}
\newcommand{\End}{\operatorname{End}}
\newcommand{\rank}{\operatorname{rank}}
\newcommand{\spn}{\operatorname{span}}

\title[Flops and about: a guide]{Flops and about: a guide}

\author[Sabin Cautis]{Sabin Cautis \thanks{The author is thankful for the support received through NSF grant DMS-1101439 and the Alfred P. Sloan foundation.}}

\begin{document}

\begin{abstract}
Stratified flops show up in the birational geometry of symplectic varieties such as resolutions of nilpotent orbits and moduli spaces of sheaves. Constructing derived equivalences between varieties related by such flops is, strangely enough, related to areas in representation theory and knot homology. In this paper we discuss how to construct such equivalences, explain the main tool for doing this (categorical Lie algebra actions) and comment on various related topics. 
\end{abstract}



\maketitle
\tableofcontents

\section{Introduction}

The stratified Mukai flop is an algebro-geometric construction relating two birational varieties. There are three basic types of Mukai flops (A, D and E6;I/E6;II) named after the Lie algebra to which they are related. Namikawa coined these terms in \cite{Nam3} where he showed that any two Springer resolutions of a nilpotent orbit closure are connected by a series of such flops (this result also appears in \cite{F}). There are also deformations of these Mukai flops which we call Atiyah flops. 

The local model of a Mukai flop of type A is a correspondence which relates the cotangent bundles $T^\star \bG(k,N)$ and $T^\star \bG(N-k,N)$ of Grassmannians. These cotangent bundles are examples of Springer resolutions of the closure of nilpotent orbits. In this case the orbit is that of a matrix $X \in GL_N$ where $X^2=0$ and the rank of $X$ is $\mbox{min}(k,N-k)$. 

Stratified flops also control the birational geometry of moduli spaces of sheaves on surfaces. In \cite{M} Markman examined Brill-Noether type stratifications of the moduli spaces of sheaves on a fixed K3 surface. He showed that two moduli spaces with Mukai vectors related by certain involutions of the Mukai lattice are related by stratified Mukai flops. 

A basic question is when two birational varieties $X$ and $Y$ have isomorphic derived categories of coherent sheaves $D(X)$ and $D(Y)$. If $D(X) \cong D(Y)$ then we say that $X$ and $Y$ are derived equivalent. One general conjecture implies that two varieties related by a stratified flop are derived equivalent. In particular, $T^\star \bG(k,N)$ and $T^\star \bG(N-k,N)$ should be derived equivalent. 

Constructing this equivalence is the motivating problem discussed in this paper. Of course, $T^\star \bG(k,N)$ and $T^\star \bG(N-k,N)$ are actually isomorphic, but this isomorphism is not canonical. The derived equivalence described in section \ref{sec:equivlocmod} is canonical in the sense that it also works in families. This means that for any vector bundle $W$ over some base one may consider the relative cotangent bundles $T^\star \bG(k,W)$ and $T^\star \bG(N-k,W)$. These might not be isomorphic (for the same reason $W$ and $W^\vee$ might not be isomorphic) but, nevertheless, are derived equivalent. 

The case $k=1$ goes back a few years to the work of Kawamata \cite{K1} and Namikawa \cite{Nam1} who constructed equivalences $D(T^\star \bG(1,N)) \xrightarrow{\sim} D(T^\star \bG(N-1,N))$. Kawamata \cite{K2} was also able to work out the case $k=2$ and $N=4$ and conjecture explicit equivalences when $k=2$ and $N > 4$. 

To deal with arbitrary $k$ and $N$ we adopted a technique used by Chuang and Rouquier \cite{CR} in modular representation theory. The idea \cite{CKL1} is to construct a categorical $\sl_2$ action (defined in section \ref{sec:catsl2action}) on the union of all $D(T^\star \bG(k,N))$ where $N$ is fixed and $k=0, \dots, N$. Recall that given an $\sl_2$ representation one can construct an action of its Weyl group which induces an isomorphism of weight spaces. Likewise, given a categorical $\sl_2$ action one can construct an action of its braid group. In this case of cotangent bundles to Grassmannians this braid group induces natural equivalences $D(T^\star \bG(k,N)) \xrightarrow{\sim} D(T^\star \bG(N-k,N))$ \cite{CKL2,CKL3}. In a sense, these equivalences unify Seidel-Thomas twists \cite{ST} and $\P^n$-twists \cite{HT} into a more general concept. 

Categorical $\sl_2$ actions have a counterpart for any Kac-Moody Lie algebra $\g$. When $\g = \sl_n$ one can define a categorical $\sl_n$ action on cotangent bundles to $n$-step partial flag varieties. This induces an action of the braid group on $n$ strands \cite{CK3} on the derived categories of these varieties which generalizes work of Khovanov-Thomas \cite{KT}, Riche \cite{Ric} and Bezrukavnikov-Riche \cite{BR}.

Stratified flops also show up in the geometry of the {affine Grassmannian}\index{affine Grassmannian}. More precisely, the twisted products $\Gr_\l \tilde{\times} \Gr_\mu$ and $\Gr_\mu \tilde{\times} \Gr_\l$ of orbits in the affine Grassmannian of $PGL_N$ are related by stratified Mukai flops (see \cite[Sec. 1]{CK2}). One can construct a geometric categorical $\sl_2$ actions here which subsequently induces an equivalence $D(\Gr_\l \tilde{\times} \Gr_\mu) \xrightarrow{\sim} D(\Gr_\mu \tilde{\times} \Gr_\l)$. 

More generally, one can define a geometric categorical $\sl_n$ action on categories of the form $D(\Gr_{\l_1} \tilde{\times} \dots \tilde{\times} \Gr_{\l_n})$ where $\l_1, \dots, \l_n$ are fundamental weights. This action induces a braid group action on these categories. In \cite{CK1,CK2,C2} these braid group actions were used to construct {homological knot invariants}\index{knot invariants} such as {Khovanov homology}\index{Khovanov homology}. 

This paper is, for the most part, a survey of topics related to flops, categorical Lie algebra actions and derived equivalences. It is an expanded version of the talk given at the conference ``Derived categories'' organized by Yujiro Kawamata and Yukinobu Toda in Tokyo in January 2011. 

\subsection{Acknowledgements}
I began working in this area with the paper \cite{CK1} jointly written with Joel Kamnitzer. In it we give an algebro-geometric construction of Khovanov knot homology. Trying, at first somewhat unsuccessfully, to generalize this construction led us, over several years, in several tangential directions. I would like to thank Joel for many years of truly fantastic ideas and insights which he shared with great enthusiasm and without which I am convinced none of this would have been possible. 

I would also like to thank Yujiro Kawamata and Yukinobu Toda for inviting me and organizing a very interesting, inspirational and enjoyable conference in Tokyo in January 2011.

\section{Stratified flops of type A -- the local model}

Here we recall the definitions of stratified Mukai flops and stratified Atiyah flops of type A. 

\subsection{Cotangent bundles to Grassmannians}\label{sec:flops1}

The local model for a {\em stratified Mukai flop of type A}\index{stratified Mukai flop of type A} is based on {cotangent bundles}\index{cotangent bundles} to Grassmannians. The cotangent bundle $T^\star \bG(k,N)$ can be described very explicitly as
\begin{equation}\label{eq:Grassdef}
\{ (X,V): X \in \End(\C^N), 0 \xrightarrow{k} V \xrightarrow{N-k} \C^N, X \C^N \subset V \text{ and } X V \subset 0 \}
\end{equation}
where $\bG(k,N)$ denotes the Grassmannian of $k$-planes in $\C^N$. The arrows denote inclusions and the superscripts indicate the codimension of the inclusion. 

We will suppose from now on that $2k \le N$. These cotangent bundles come equipped with the affinization map 
$$p(k,N): T^\star \bG(k,N) \rightarrow \overline{B(k,N)}$$ 
where 
$$B(k,N) := \{ X \in \End(\C^N) : X^2 = 0 \text{ and } \dim(\ker X) = N-k \}$$
and $p$ is the map which forgets $V$. This map is birational since generically the rank of $X$ is $k$ and one can recover $V$ as the image of $X$. Likewise there is a projection map 
$$p(N-k,N): T^\star \bG(N-k,N) \rightarrow \overline{B(k,N)}$$ 
which is also birational since one can recover $V$ as the kernel of $X$. The triple 
\begin{equation}\label{eq:10}
\xymatrix{
T^* \bG(k,N) \ar[dr]_{p(k,N)} & & T^* \bG(N-k,N) \ar[dl]^{p(N-k,N)} \\
& \overline{B(k,N)} &
}
\end{equation}
is the local model for a stratified Mukai flop of type A. When $k=1$ this is the usual (and better known) Mukai flop. 

\subsection{Deformations of cotangent bundles}\label{sec:flops2}

The cotangent bundles above have a natural one-parameter deformation $\wt{T^\star \bG(k,N)}$ over $\A^1$. This deformation can be described explicitly as 
\begin{align}\label{eq:defGrassdef}
\{(X,V,x): & X \in \End(\C^N), 0 \subset V \subset \C^N, x \in \C, \dim(V) = k \nonumber  \\ 
& \text{ and } X \C^N \subset V,  (X-x \cdot \id) V \subset 0 \}
\end{align}
where the map to $\A^1$ remembers $x$. The fibre over $x=0$ is clearly just $T^\star \bG(k,N)$. These varieties also come equipped with the affinization map 
$$\tp(k,N): T^\star \bG(k,N) \rightarrow \overline{\tB(k,N)}$$ 
where $\tB(k,N)$ is the variety
$$\{(X,x): X \in \End(\C^N), x \in \C, X(X- x \cdot \id) = 0 \text{ and } \dim (\ker (X - x \cdot \id)) = k \}$$
and $\tp(k,N)$ forgets $V$. The map is again birational. In fact, it is an isomorphism if $x \ne 0$ because one can recover $X$ as the kernel of $(X-x \cdot \id)$. The diagram
\begin{equation*}
\xymatrix{
\wt{T^* \bG(k,N)} \ar[dr]^{\tp(k,N)} & & \wt{T^* \bG(N-k,N)} \ar[dl]_{\tp(N-k,N)} \\
& \overline{\tB(k,N)} \cong \overline{\tB(N-k,N)} &}
\end{equation*}
is the local model for a {\em stratified Atiyah flop}\index{stratified Atiyah flop} of type A. The isomorphism $\overline{\tB(k,N)} \cong \overline{\tB(N-k,N)}$ is given by $(X,x) \mapsto (X - x \cdot \id, -x)$. When $k=1$ and $N=2$ this is the usual Atiyah flop (hence the terminology) where both deformed cotangent bundles turn out to be isomorphic to the total space of the vector bundle $\O_{\P^1}(-1) \oplus \O_{\P^1}(-1)$ over $\P^1$. 

\subsubsection{$\C^\times$-actions}
There are compatible $\C^\times$-actions on $T^\star \bG(k,N)$ and its deformation defined by
$$t \cdot (X,V) = (t^2 X,V) \text{ and } t \cdot (X,V,x) = (t^2 X, V, t^2 x)
$$
respectively. Also, notice that both $T^\star \bG(k,N)$ and its deformation carry a tautological bundle, denoted $V$, whose fibre over $(X,V)$ (or $(X,V,x)$) is $V$.

\section{Geometric categorical $\sl_2$ actions}\label{sec:catsl2action}

The main tool used to construct derived equivalences between flops is the notion of a {\em geometric categorical $\sl_2$ action}\index{geometric categorical ${\mathfrak{sl}}_2$ action}. The idea of using categorical $\sl_2$ actions originates with Chuang and Rouquier's remarkable proof of Brou\'e's abelian defect group conjecture for symmetric groups \cite{CR}. They construct and use such an action to define equivalences between blocks of representations of the symmetric group in positive characteristic. We adapt their approach to categories of coherent sheaves. 

\subsection{Preliminary concepts} 

All varieties are defined over $\C$. If $X$ is a variety we denote by $D(X)$ the {\em bounded derived category of coherent sheaves}\index{derived category of coherent sheaves} on $X$. As usual, we denote by $[1]$ the cohomological shift in $D(X)$ downwards by $1$. 

\subsubsection{Fourier-Mukai transforms}
An object $\sP \in D(X \times Y)$ whose support is proper over $Y$ induces a {\em Fourier-Mukai}\index{Fourier-Mukai} (FM) functor $\Phi_{\sP}: D(X) \rightarrow D(Y)$ via $(\cdot) \mapsto \pi_{2*}(\pi_1^* (\cdot) \otimes \sP)$ (where every operation is derived). One says that $\sP$ is the {\em FM kernel}\index{Fourier-Mukai kernel} which induces $\Phi_{\sP}$. The right and left adjoints $\Phi_{\sP}^R$ and $\Phi_{\sP}^L$ are induced by $\sP_R := \sP^\vee \otimes \pi_2^* \omega_X [\dim(X)]$ and $\sP_L := \sP^\vee \otimes \pi_1^* \omega_Y [\dim(Y)]$ respectively.

If $\sQ \in D(Y \times Z)$ then $\Phi_{\sQ}  \Phi_{\sP} \cong \Phi_{\sQ * \sP}: D(X) \rightarrow D(Z)$ where $\sQ * \sP = \pi_{13*}(\pi_{12}^* \sP \otimes \pi_{23}^* \sQ)$ is the convolution product. So instead of talking about functors and compositions we will speak of kernels and convolutions.

\subsubsection{$\C^\times$-equivariance}
If $X$ carries a $\C^\times$-action then we will consider the bounded derived category of $\C^\times$-equivariant coherent sheaves on $X$ which, abusing notation, we also denote by $D(X)$. The sheaf $\O_X\{i\}$ denotes the structure sheaf of $X$ shifted with respect to the $\C^\times$-action so that if $f \in \O_X(U)$ is a local function then viewed as a section $f' \in \O_X\{i\}(U)$ we have $t \cdot f' = t^{-i}(t \cdot f)$. We denote by $\{i\}$ the operation of tensoring with $\O_X\{i\}$. 

Since $D(X)$ carries a grading $\{\cdot\}$ its Grothendieck group is actually a $\Z[q,q^{-1}]$-module where $-q$ acts by twisting by $\{1\}$. We usually tensor the Grothendieck group with $\C$ so that it becomes a $\C[q,q^{-1}]$-module and denote it $K(X)$.

\subsubsection{Convolution of complexes}\label{sec:convcpx}
Now consider a complex 
$$\sP_\bullet := \left[ \sP_m \xrightarrow{d} \sP_{m-1} \xrightarrow{d} \dots \xrightarrow{d} \sP_1 \xrightarrow{d} \sP_0 \right]$$
where $\sP_i \in D(X)$ and $d^2=0$. If $m=1$ one can just take the cone and obtain an object in $D(X)$. If $m > 1$ one would like to take an iterated cone. This is commonly called a right or a left {\em convolution}\index{convolution of complexes} of $\sP_\bullet$ depending on whether you start the iterated cone from the right end or from the left end. Do not confuse this convolution with the convolution of kernels described above! 

In general a right convolution is not guaranteed to exist or to be unique. This is because the $\Cone$ operation is not functorial. However, under the following cohomological conditions 
$$\Hom(\sP_{i+k+1}[k], \sP_i) = 0 \text{ and } \Hom(\sP_{i+k+2}[k], \sP_i) = 0 \text{ for } i \ge 0, k \ge 1$$
a unique right convolution exists. For details see \cite[Sec. 3.4]{CKL3}.

\subsection{Definition}\label{sec:def}

Let us recall the definition of a geometric categorical $\sl_2$ action from \cite{CKL1}. To shorten notation we will write $H^\star(\P^r)$ for the (doubly) graded vector space
\begin{equation*}
\C[r]\{-r\} \oplus \C[r-2]\{-r+2\} \oplus \dots \oplus \C[-r+2]\{r-2\} \oplus \C[-r]\{r\}.
\end{equation*}
By convention $H^\star(\P^{-1})$ is zero.

A geometric categorical $\sl_2$ action consists of the following data.

\begin{enumerate}
\item A collection of smooth complex varieties $Y(\l)$ indexed by $\l \in \Z$ and equipped with $\C^\times$-actions.
\item Fourier-Mukai kernels
$$\sE^{(r)}(\l) \in  D(Y(\l-r) \times Y(\l+r)) \text{ and } \sF^{(r)}(\l) \in  D(Y(\l+r) \times Y(\l-r))$$
(which are $\C^\times$ equivariant).
We write $\sE(\l)$ for $\sE^{(1)}(\l)$ and $\sF(\l)$ for $\sF^{(1)}(\l)$ while $\sE^{(0)}(\l)$ and $\sF^{(0)}(\l)$ are equal to the identity kernels $\O_\Delta$.
\item For each $Y(\l)$ a flat deformation $\tY(\l) \rightarrow \A^1$ carrying a $\C^\times$-action compatible with the action $x \mapsto t^2 x$ (where $t \in \C^\times$) on the base $\A^1$.
\end{enumerate}

On this data we impose the following additional conditions.

\begin{enumerate}
\item $Y(\l) = \emptyset$ for $\l \gg 0$ or $\l \ll 0$. Moreover, each (graded piece of the) $\Hom$ space between two objects in $D(Y(\l))$ is finite dimensional. In particular, this means that if $Y(\l) \ne \emptyset$ then $\End(\O_{Y(\l)}) = \C \cdot I$.

\item All $\sE^{(r)}$s and $\sF^{(r)}$s are sheaves (i.e. complexes supported in degree zero).

\item
$\sE^{(r)}(\l)$ and $\sF^{(r)}(\l)$ are left and right adjoints of each other up to shift. More precisely
\begin{enumerate}
\item $\sE^{(r)}(\l)_R = \sF^{(r)}(\l)[r\l]\{-r\l\}$ and $\sF^{(r)}(\l)_L = \sE^{(r)}(\l)[r\l]\{-r\l\}$
\item $\sE^{(r)}(\l)_L = \sF^{(r)}(\l)[-r\l]\{r\l\}$ and $\sF^{(r)}(\l)_R = \sE^{(r)}(\l)[-r\l]\{r\l\}$
\end{enumerate}

\item
At the level of cohomology of complexes we have
$$\sH^*(\sE(\l+r) * \sE^{(r)}(\l-1)) \cong \sE^{(r+1)}(\l) \otimes_{\C} H^\star(\P^{r}).$$

\item
If $ \l \le 0 $ then
\begin{equation*}
\sF(\l+1) * \sE(\l+1) \cong \sE(\l-1) * \sF(\l-1)  \oplus \sP
\end{equation*}
where $\sH^*(\sP) \cong \O_\Delta \otimes_\C H^\star(\P^{-\l-1})$.

Similarly, if $\l \ge 0$ then
\begin{equation*}
\sE(\l-1) * \sF(\l-1) \cong \sF(\l+1) * \sE(\l+1) \oplus \sP'
\end{equation*}
where $\sH^*(\sP') \cong \O_\Delta \otimes_\C H^\star(\P^{\l-1})$.

\item  We have
$$\sH^*(i_{23*} \sE(\l+1) * i_{12*} \sE(\l-1)) \cong \sE^{(2)}(\l)[-1]\{1\} \oplus \sE^{(2)}(\l)[2]\{-3\}$$
where the $i_{12}$ and $i_{23}$ are the closed immersions
\begin{align*}
i_{12}: Y(\l-2) \times Y(\l) \rightarrow Y(\l-2) \times \tY(\l) \\i_{23}: Y(\l) \times Y(\l+2) \rightarrow \tY(\l) \times Y(\l+2).
\end{align*}

\item If $\l \le 0$ and $k \ge 1$ then the image of $\mbox{supp}(\sE^{(r)}(\l-r))$ under the projection to $Y(\l)$ is not contained in the image of $\mbox{supp}(\sE^{(r+k)}(\l-r-k))$ also under the projection to $Y(\l)$. Similarly, if $\l \ge 0$ and $k \ge 1$ then the image of $\mbox{supp}(\sE^{(r)}(\l+r))$ in $Y(\l)$ is not contained in the image of $\mbox{supp}(\sE^{(r+k)}(\l+r+k))$.

\end{enumerate}

At the level of Grothendieck groups $\sE$ and $\sF$ induce maps of $\C$-vector spaces
$$E: K(Y(\l-1)) \rightarrow K(Y(\l+1)) \text{ and } F: K(Y(\l+1)) \rightarrow K(Y(\l-1)).$$
This gives an action of $\sl_2$ on $\oplus_\l K(Y(\l))$ where the weight spaces are $K(Y(\l))$. In fact, everything is over $\C[q,q^{-1}]$ and we actually obtain a $U_q(\sl_2)$ representation. So the above action should really be called a geometric categorical $U_q(\sl_2)$ action. 

\subsection{Some remarks}\label{sec:remarks}

The definition above is not necessarily the simplest but is tailored so that it is easier to check on categories of coherent sheaves. Here are some remarks about the relevance of conditions (i) through (vii) above. 

Condition (i) is used to ensure that the Krull-Schmidt property holds (namely unique decomposition into irreducibles). Condition (ii) is used to make sense of condition (iv). 

Conditions (iv) and (v) are checked only at the level of cohomology. This is because it is often possible to compute the cohomology of an object (like $\sP$ in condition (v)) but difficult to show that the object is formal (i.e. the direct sum of its cohomology). The role of the deformation $\tY(\l) \rightarrow \A^1$ is actually related to this issue. We explain this now. 

The short exact sequence of tangent bundles 
$$0 \rightarrow T_{Y(\l)} \rightarrow T_{\tY(\l)}|_{Y(\l)} \rightarrow N_{Y(\l)/\tY(\l)} \rightarrow 0$$
gives us a connecting map $b \in H^1(Y(\l), T_{Y(\l)} \{-2\})$ since $N_{Y(\l)/\tY(\l)} \cong \O_{Y(\l)} \{2\}$. This is just the first order deformation corresponding to $\tY(\l) \rightarrow \A^1$ and is uniquely defined up to a non-zero multiple. Now, the Hochschild-Kostant-Rosenberg isomorphism states that 
$$\Delta^* \Delta_* \O_{Y(\l)} \cong \bigoplus_i \wedge^i T^\star_{Y(\l)} [i]$$
where $\Delta: Y(\l) \rightarrow Y(\l) \times Y(\l)$ is the inclusion as the diagonal. This implies that 
\begin{align}
& \Hom(\Delta_* \O_{Y(\l)}, \Delta_* \O_{Y(\l)} [2]) \nonumber \\
\cong & H^0(Y(\l), \wedge^2 T_{Y(\l)}) \oplus H^1(Y(\l), T_{Y(\l)}) \oplus H^2(Y(\l), \O_{Y(\l)}). \nonumber
\end{align}
In particular, this means that $b$ induces a map 
\begin{equation}\label{eq:7}
\beta: \Delta_* \O_{Y(\l)} \rightarrow \Delta_* \O_{Y(\l)} [2]\{-2\}.
\end{equation}
Unfortunately, in practice it is difficult to get your hands on such a map. The purpose of the deformation $\tY(\l)$ is simply to yield $\beta$. 

Now consider the map 
$$I \beta I: \sE * \Delta_* \O_{Y(\l)} * \sE \rightarrow \sE * \Delta_* \O_{Y(\l)} * \sE [2] \{-2\}$$
where $\sE * \sE \in D(Y(\l-2) \times Y(\l+2))$. The cohomology $\sH^{-1}$ of both sides is $\sE^{(2)} \{-1\}$. The content of condition (vi) is that the map above induces an isomorphism on $\sH^{-1}$. In turn, this allows you to conclude that $\sE * \sE$ equals $\sE^{(2)} [-1]\{1\} \oplus \sE^{(2)} [1]\{-1\}$ on the nose rather than at the level of cohomology. This is by a little trick that goes back at least to Deligne. For more details see \cite{CKL2}. 

Finally, condition (vii) is an annoying technical condition which is only ever used once (namely in Lemma 4.6 of \cite{CKL2} which is itself technical in nature). Though unsightly, its main advantage is that it is very easy to check.

\subsection{Inducing equivalences}\label{sec:equivalences}

We first explain why all this is related to constructing {\em equivalences}\index{derived equivalences}. Suppose one has an $\sl_2$ action on a vector space $V$. The action of $H := \left( \begin{matrix} 1 & 0 \\ 0 & -1 \end{matrix} \right) \in \sl_2$ breaks up $V$ into $H$-eigenspaces $V(\l)$ where $Hv = \l v$ if $v \in V(\l)$. Moreover, using the relation $[E,F] = H$ where 
$$E := \left( \begin{matrix} 0 & 1 \\ 0 & 0 \end{matrix} \right) \in \sl_2 \text{ and } 
F := \left( \begin{matrix} 0 & 0 \\ 1 & 0 \end{matrix} \right) \in \sl_2$$
one can check that $E: V(\l) \rightarrow V(\l+2)$ and $F: V(\l) \rightarrow V(\l-2)$. 

If $V(\l)=0$ for $\l \gg 0$ or $\l \ll 0$ then this action integrates to an action of the Lie group $SL_2(\C)$. Here it is known that the reflection element
$$t = \left[ \begin{matrix} 0 & -1 \\ 1 & 0 \end{matrix} \right] $$
induces an isomorphism of vector spaces $V(\l) \xrightarrow{\sim} V(-\l)$. Moreover, if say $\l \ge 0$, we can write $t$ as 
\begin{equation}\label{eq:t}
t = F^{(\l)} - F^{(\l+1)}E + F^{(\l+2)}E^{(2)} \pm \dots
\end{equation}
where $E^{(k)} := E^k/k!$ and $F^{(k)} := F^k/k!$. Notice that the sum is finite since $V(\l)=0$ for $\l \gg 0$.

Now we try to immitate this construction with categories. We replace $V(\l)$ by the category $D(Y(\l))$, the functors $E^{(r)}$ and $F^{(r)}$ by the kernels $\sE^{(r)}$ and $\sF^{(r)}$ and the sum (\ref{eq:t}) describing $t$ with a complex 
$$\Theta_* = \left[ \dots \rightarrow \Theta_s \rightarrow \Theta_{s-1} \rightarrow \dots \rightarrow \Theta_1 \rightarrow \Theta_0 \right]$$
where 
\begin{equation}\label{eq:9}
\Theta_s := \sF^{(\l+s)}(s) * \sE^{(s)}(\l+s) [-s]\{s\} \in D(Y(\l) \times Y(-\l)).
\end{equation}
Again, this complex is finite since $Y(\l)$ is empty for $\l \gg 0$ or $\l \ll 0$. The differential is given by the composition 
$$\sF^{(\l+s+1)} * \sE^{(s+1)} \rightarrow \sF^{(\l+s)} * \sF * \sE * \sE^{(s)} [-\l-2s]\{\l+2s\} \rightarrow \sF^{(\l+s)} * \sE^{(s)} [1]\{-1\}$$
where the first map is the inclusion of $\sF^{(\l+s+1)}$ and $\sE^{(s+1)}$ into the lowest cohomological degrees of $\sF^{(\l+s)} * \sF$ and $\sE * \sE^{(s)}$ respectively while the second map is induced by the adjunction map $\sF * \sE \rightarrow \O_\Delta [\l+2s+1] \{-\l-2s-1\}$ (using that $\sF$ is the left adjoint of $\sE$ up to a shift). The complex $\Theta_*$ is sometimes called Rickard's complex. 

\begin{theorem}\cite[Thm. 2.8]{CKL3}\label{thm:main1}
The complex $\Theta_*$ has a unique right convolution $\sT(\l) \in D(Y(\l) \times Y(-\l))$. Moreover, $\Phi_{\sT(\l)}: D(Y(\l)) \xrightarrow{\sim} D(Y(-\l))$ is an equivalence which categorifies the isomorphism $t: K(Y(\l)) \xrightarrow{\sim} K(Y(-\l))$.
\end{theorem}

This theorem is proved in two steps. In the first step \cite{CKL2} we prove that a geometric categorical $\sl_2$ action induces a strong categorical $\sl_2$ action. Without recalling the precise definition of the latter let us note that its most remarkable property is an action of the nilHecke algebra on $\sE$'s. 

More precisely, in \cite{CKL2} we show that given a geometric categorical $\sl_2$ action one can construct two types of maps 
$$X: \sE \rightarrow \sE [2]\{-2\} \text{ and }
T: \sE * \sE \rightarrow \sE * \sE [-2]\{2\}$$
which satisfy the following relations
\begin{enumerate}
\item $T^2=0$ where $T \in \End(\sE * \sE)$,
\item $(IT)(TI)(IT)=(TI)(IT)(TI)$ where $TI,IT \in \End(\sE * \sE * \sE)$,
\item $(XI)T-T(IX)=I=-(IX)T+T(XI)$ where $XI,IX,T \in \End(\sE * \sE)$.
\end{enumerate}
Recall that $\sE$ and $\sE * \sE$ are just (complexes of) sheaves so $X$ and $T$ are maps of (complexes of) sheaves. If instead we think of the functors induced by $\sE$ and $\sE * \sE$ then $X$ and $T$ are natural transformations of functors. 

In the second step \cite{CKL3} we show that in a strong categorical $\sl_2$ the complex $\Theta_*$ has a unique right convolution which induces an equivalence. The role of the nilHecke algebra is to help simplify expressions of the form $\Theta_* * \sF^{(r)}$. 

Of course, the second step no longer involves any geometry. In fact, a similar result was proved in \cite{CR}. However, their action was on abelian categories and it was not clear how to extend it to triangulated categories. In the end, the proof we give in \cite{CKL3} is fairly different from that in \cite{CR}.  

The maps $X$ and $T$ are examples of higher structure in the representation theory of $\sl_2$. The r\^ole of the nilHecke algebra in the (higher) representation theory of $\sl_2$ is studied in detail by Lauda in \cite{L}. Subsequently, Khovanov-Lauda \cite{KL1,KL2,KL3} and Rouquier \cite{Rnew} describe certain graded algebras now called quiver Hecke algebras or KLR algebras which play the r\^ole for other Lie algebras (such as $\sl_m)$ that the nilHecke plays for $\sl_2$. We will discuss certain categorical $\sl_m$ actions in section \ref{sec:slm} although we do not make any further reference to these KLR algebras. 

\section{Equivalences for the local model of stratified flops}\label{sec:equivlocmod}

\subsection{Categorical actions on $\oplus_k D(T^\star \bG(k,N))$}\label{sec:action}

In this section we fix $N$ and let 
\begin{equation}\label{eq:1}
Y(\l) := T^\star \bG(k,N) \text{ and }
\tY(\l) := \wt{T^\star \bG(k,N)} \text{ where } 
\l = N-2k.
\end{equation}
Consider the correspondences 
$$ W^r(\l) \subset Y(\l-r) \times Y(\l+r) = T^\star \bG(k+r/2, N) \times T^\star \bG(k-r/2, N) $$
defined by
\begin{align*}
W^r(\l) := \{ (X,V,V') : &X \in \End(\C^N), \dim(V) = k + \frac{r}{2}, \dim(V') = k - \frac{r}{2}, \\
& 0 \subset V' \subset V \subset \C^N,  \C^N \xrightarrow{X} V' \text{ and } V \xrightarrow{X} 0 \}.
\end{align*}
There are two natural projections $\pi_1: (X,V,V') \mapsto (X,V)$ and $\pi_2: (X,V,V') \mapsto (X,V')$ from $W^r(\l)$ to $Y(\l-r)$ and $Y(\l+r)$ respectively. Together they give us an embedding
$$(\pi_1, \pi_2): W^r(\l) \subset Y(\l-r) \times Y(\l+r).$$

On $W^r(\l)$ there are two tautological bundles, namely $V := \pi_1^*(V)$ and $V' := \pi_2^*(V)$ where the prime on the $V'$ indicates that the vector bundle is the pullback of the tautological bundle by the second projection. We also have natural inclusions
$$0 \subset V' \subset V \subset \C^N$$
where $\C^N$ denotes the trivial vector bundle on $W^r(\l)$.

Now define kernels 
$$\sE^{(r)}(\l) \in  D(Y(\l-r) \times Y(\l+r)) \text{ and } \sF^{(r)}(\l) \in D(Y(\l+r) \times Y(\l-r))$$
by
\begin{align}
\label{eq:2}
\sE^{(r)}(\l) &:= \O_{W^r(\l)} \otimes \det(\C^N/V)^{-r} \otimes \det(V')^r \{r(N-\l-r)/2\} \\
\label{eq:3}
\sF^{(r)}(\l) &:= \O_{W^r(\l)} \otimes \det(V/V')^{\l} \{r(N+\l-r)/2\}.
\end{align}

In \cite{CKL3} (although most of the hard work is done in \cite{CKL1}) we prove the following:

\begin{theorem}\cite[Thm. 6.1]{CKL3} \label{thm:main2}
The varieties $Y(\l)$ and their deformations $\tY(\l)$ defined in (\ref{eq:1}) together with the functors $\sE^{(r)}(\l)$ and $\sF^{(r)}(\l)$ from (\ref{eq:2}) and (\ref{eq:3}) define a geometric categorical $\sl_2$ action. 
\end{theorem}

So, as a consequence of Theorem \ref{thm:main1}, this gives us an equivalence
$$\Phi_{\sT(k,N)}: D(T^\star \bG(k,N)) \xrightarrow{\sim} D(T^\star \bG(N-k,N)).$$
In fact, one can show that $\sT(k,N)$ is a sheaf (\cite{CKL2} Prop. 6.6). In the next section we identify $\sT(k,N)$ more explicitly. 

\subsection{The {equivalence}\index{derived equivalences}: an explicit description}

\subsubsection{Some geometry}\label{sec:somegeometry}
Using the above notation $Y(\l) := T^\star \bG(k,N)$ where $\l = N-2k$ recall that the stratified Mukai flop is summarized by the diagram 
$$Y(\l) \xrightarrow{p(k,N)} \overline{B(k,N)} \xleftarrow{p(N-k,N)} Y(-\l).$$
Now consider the fibre product
\begin{align}
Z(k,N) &:= Y(\l) \times_{\overline{B(k,N)}} Y(-\l) \nonumber \\
&= \{ 0 \overset{k}{\underset{N-k}{\rightrightarrows}} \begin{matrix} V \\ V' \end{matrix} \overset{N-k}{\underset{k}{\rightrightarrows}} \C^N : X \C^N \subset V, X \C^N \subset V', XV \subset 0, XV' \subset 0 \}. \nonumber
\end{align}
Since $p(k,N)$ and $p(N-k,N)$ are semi-small, $Z(k,N)$ is equidimensional of dimension $2k(N-k)$. It consists of $(k+1)$ irreducible components $Z_s(k,N)$ ($s = 0, \dots, k)$ where
\begin{equation*}
Z_s(k,N) := \overline{p(k,N)^{-1}(B(k-s,N)) \times_{B(k-s,N)} p(N-k,N)^{-1}(B(k-s,N))}.
\end{equation*}
The component $Z_s(k,N)$ can be described more directly as 
$$\{(X,V,V') \in Z(k,N): \dim (\ker X) \ge N-k+s \text{ and } \dim (V \cap V') \ge k-s \}.$$
It is helpful to keep in mind the following. Any two components $Z_s(k,N)$ and $Z_{s'}(k,N)$ intersect in a divisor if $|s-s'|=1$ but their intersection has strictly higher codimension if $|s-s'| > 1$. 

Now, since $\spn(V,V') \subset \ker X$, it follows that $\dim (\ker X) + \dim (V \cap V') \ge N$ on $Z(k,N)$. We define the open subscheme
$$Z^o(k,N) := \{(X,V,V') \in Z(k,N): N+1 \ge \dim (\ker X) + \dim (V \cap V') \} \subset Z(k,N)$$
and $Z^o_s(k,N) := Z_s(k,N) \cap Z^o(k,N)$. 

\begin{theorem}\cite[Thm. 3.8]{C1} \label{thm:main3}
There exists a $\C^\times$-equivariant line bundle $\sL(k,N)$ on $Z^o(k,N)$ such that $\sT(k,N) \cong i_* j_* \sL(k,N)$ where $i$ and $j$ are the natural inclusions
$$Z^o(k,N) \xrightarrow{j} Z(k,N) \xrightarrow{i} Y(\l) \times Y(-\l).$$
\end{theorem}

Note that the map $j$ in Theorem (\ref{thm:main3}) is an open immersion. Whenever we have an open immersion in this paper $j_*$ denotes the {\it non-derived} push-forward. This is the only case in this paper when a functor is not derived. 

The line bundle $\sL(k,N)$ is uniquely determined by its restriction to each $Z^o_s(k,N)$. One has that $\sL(k,N)|_{Z^o_s(k,N)}$ is isomorphic to  
$$\O_{Z^o_s(k,N)}([D^o_{s,+}(k,N)]) \otimes \det(\C^N/V)^{-s} \otimes \det(V')^s \{k(N-k) - (k-s)^2 + s\}$$
where $D^o_{s,+}(k,N)$ is the divisor $Z_s^o(k,N) \cap Z_{s+1}(k,N)$ and $V,V'$ are the tautological bundles on $Z^o_s(k,N) \subset Y(\l) \times Y(-\l)$ pulled back from $Y(\l)$ and $Y(-\l)$ respectively. 
\subsubsection{Why is Theorem \ref{thm:main3} believable?}

Recall that $\sT(k,N)$ is the right convolution of the complex 
\begin{equation}\label{eq:4}
\sF^{(\l+k)} * \sE^{(k)} [-k]\{k\} \rightarrow \dots \rightarrow \sF^{(\l+1)} * \sE [-1] \{1\} \rightarrow \sF^{(\l)}.
\end{equation}
Now, for $s=0,1, \dots, k$, one can show \cite[Prop. 6.3]{CKL3} that $\sF^{(\l+s)} * \sE^{(s)}$ is a sheaf supported exactly on $Z_s(k,N)$ (actually, one can identify this sheaf explicitly). It then follows quite easily that the convolution in (\ref{eq:4}) above is also a sheaf supported exactly on $\bigcup_s Z_s(k,N) = Z(k,N)$. 

So we just need to identify this sheaf. To do this we first argue that $\sT(k,N)$ is (the push-forward by a closed embedding of) a Cohen-Macaulay sheaf. This is done by identifying the kernel which is the inverse of $\sT(k,N)$ (see section \ref{sec:inverse}) and showing by the same argument above that it is also a sheaf. Formal non-sense says that the inverse kernel is just $\sT(k,N)^\vee$ tensored with some line bundle and a shift. This means $\sT(k,N)^\vee$ is a sheaf and hence $\sT(k,N)$ is (the push-forward of) a Cohen-Macaulay sheaf.

Finally, any Cohen-Macaulay sheaf is uniquely determined by its restriction to an open subset of codimension at least two. The last step is to identify the restriction of $\sT(k,N)$ to $Z^o(k,N)$ which is codimension two inside $Z(k,N)$ (we do this in \cite{C1}). The advantage of $Z^o(k,N)$ over $Z(k,N)$ is that two components $Z_s^o(k,N)$ and $Z_{s'}^o(k,N)$ in $Z^o(k,N)$ intersect in a Cartier divisor if $|s-s'|=1$ and are {\it disjoint} if $|s-s'|>1$. So $Z^o(k,N)$ avoids all the nastier singularities of $Z(k,N)$. 

\subsection{The inverse}\label{sec:inverse}

The inverse $\sT(k,N)^{-1}$ of $\sT(k,N)$ is given by its left (or equivalently its right) adjoint. This is equal to the left convolution of the complex
$$(\Theta_0)_L \rightarrow (\Theta_1)_L \rightarrow \dots \rightarrow (\Theta_{s-1})_L \rightarrow (\Theta_{s})_L \rightarrow \dots$$
where it is easy to check \cite[Sec. 5.1]{C1} that 
$$(\Theta_s)_L \cong \sF^{(s)} * \sE^{(N-2k+s)} [s] \{ -s \}.$$
It then follows that $\sT(k,N)^{-1}$ is again a sheaf which is the push-forward of a line bundle from an open subset of $Z(k,N)$ \cite[Thm. 5.3]{C1}. Perhaps a little more surprising:

\begin{proposition}\cite[Cor. 7.5]{CKL4}\label{prop:main2}
The kernels $\sT(k,N)^{-1}$ and $\sT(N-k,N)$ are related by
$$\sT(k,N)^{-1} \cong \sT(N-k,N) \otimes \det(V)^{-1} \otimes \det(V')^{-1} \{N-2k\} \in D(Y(-\l) \times Y(\l)).$$
\end{proposition}

This isomorphism is something which is special to the example of cotangent bundles of Grassmannians. In other words, such a relation is not a formal consequence of having a geometric categorical $\sl_2$ action.

\subsection{The equivalence: stratified Atiyah flops}

Recall the notation $\tY(\l) := \wt{T^\star \bG(k,N)}$ which comes equipped with a map $\tY(\l) \rightarrow \A^1$. Similar to the above description of $\sT(k,N)$ we now proceed to describe a kernel $\tilde{\sT}(k,N)$ which induces an equivalence $\Phi_{\tilde{\sT}(k,N)}: D(\tY(\l)) \xrightarrow{\sim} D(\tY(-\l))$. 

\subsubsection{Some geometry}
Recall that the stratified Atiyah flop is summarized by the diagram
$$\tY(\l) \xrightarrow{\tp(k,N)} \overline{\tB(k,N)} \cong \overline{\tB(N-k,N)} \xleftarrow{\tp(N-k)} \tY(-\l).$$
Once again we can consider the fibre product
\begin{align}
\tZ(k,N) &:= \tY(\l) \times_{\overline{\tB(k)}} \tY(-\l) \nonumber \\
&= \{ 0 \overset{k}{\underset{N-k}{\rightrightarrows}} \begin{matrix} V \\ V' \end{matrix} \overset{N-k}{\underset{k}{\rightrightarrows}} \C^N : X \C^N \subset V, (X - x \cdot \id) \C^N \subset V' \nonumber \\
& (X - x \cdot \id) V \subset 0, XV' \subset 0 \}\nonumber
\end{align}
which deforms the old fibre product $Z(k,N)$. However, unlike $Z(k,N)$, $\tZ(k,N)$ is now irreducible. Notice that $\tZ(k,N)$ is naturally a subscheme of $\tY(\l) \times_{\A^1} \tY(-\l)$ (where the second projection $\tY(-\l) \rightarrow \A^1$ maps $(X,V,x) \mapsto -x$).

Next, as before, we can define an open subscheme $\tZ^o(k,N)$ as follows
$$\{(X,V,V',x) \in \tZ(k,N): N+1 \ge \dim (\ker(X - x \cdot \id)) + \dim (V \cap V') \} \subset \tZ(k,N).$$
Notice that if $x \ne 0$ then $V$ and $V'$ are uniquely determined by $X$ as the kernels of $(X-x \cdot \id)$ and $X$ respectively. This means that $\tZ^o(k,N)$ contains all the fibres over $x \ne 0$ since when $x \ne 0$ we have $V \cap V' = 0$. 

Inside $\tZ^o(k,N)$ we have $Z^o_s(k,N)$, the components of the central fibre. One can check that these are Cartier divisors.  

\begin{theorem}\cite[Thm. 4.1]{C1}\label{thm:main4}
Consider the line bundle on $\tZ^o(k,N)$ given by
$$\tilde{\sL}(k,N) := \O_{\tZ^o(k,N)}(\sum_{s=0}^k \binom{s+1}{2} [Z^o_s(k,N)]) \{k(N-2k)\}$$ 
and let 
$$\tilde{\sT}(k,N) := \tilde{i}_* \tilde{j}_* \tilde{\sL}(k,N) \in D(\tY(\l) \times_{\A^1} \tY(-\l))$$
where $\tilde{i}$ and $\tilde{j}$ are the natural inclusions 
$$\tZ^o(k,N) \xrightarrow{\tilde{j}} \tZ(k,N) \xrightarrow{\tilde{i}} \tY(\l) \times_{\A^1} \tY(-\l).$$
Then $\Phi_{\tilde{\sT}(k,N)}: D(\tY(\l)) \xrightarrow{\sim} D(\tY(-\l))$ is an equivalence and the restriction of $\tilde{\sT}(k,N)$ to $Y(\l) \times Y(-\l)$ is $\sT(k,N)$.
\end{theorem}

The map $\tilde{j}$ above is an open immersion so again $\tilde{j}_*$ denotes the {\it non-derived} push-forward. There is no categorical $\sl_2$ action on $\oplus_{\l} D(\tY(\l))$ so to prove Theorem \ref{thm:main4} one guesses the expression for $\tilde{\sT}(k,N)$ and shows that it restricts to an equivalence over each fibre of the map to $\A^1$. Hence Theorem \ref{thm:main4} is essentially a corollary of Theorem \ref{thm:main3}. 

\subsubsection{The case $k=1$.}
The central fibre of $\tZ(1,N)$ contains two components $Z_0(1,N)$ and $Z_1(1,N)$ while
$$\tilde{\sT}(1,N) \cong \tilde{i}_* \tilde{j}_* \O_{\tZ^o(1,N)}([Z^o_1(1,N)]) \{N-2\}.$$
Now, one can show \cite[Lem. 5.1]{C1} that
\begin{equation}\label{eq:6}
\tilde{i}_* \tilde{j}_* \O_{\tZ^o(k,N)}(\sum_{s=1}^k s[Z_s^o(k,N)]) \cong \O_{\tZ(k,N)} \otimes \det(\C^N/V)^\vee \otimes \det(V') \{2k\}
\end{equation}
for any $k$. When $k=1$ this means that 
$$\tilde{\sT}(1,N) \cong \O_{\tZ(1,N)} \otimes \det(\C^N/V)^\vee \otimes \det(V') \{N\}.$$
In particular this implies:
\begin{corollary}\label{cor:main1}
The fibre product correspondences $Z(1,N)$ and $\tZ(1,N)$ induce equivalences 
$$\Phi_{\O_{Z(1,N)}}: D(Y(1)) \xrightarrow{\sim} D(Y(N-1)) \text{ and } 
\Phi_{\O_{\tZ(1,N)}}: D(\tY(1)) \xrightarrow{\sim} D(\tY(N-1)).$$
\end{corollary}

This corollary was originally proved by Kawamata \cite{K1} and Namikawa \cite{Nam1}. Namikawa \cite[Sec. 2]{Nam1} also shows that the correspondence 
\begin{equation*}
\xymatrix{
& W \ar[dl]_{\pi_1} \ar[dr]^{\pi_2} & \\
Y(1) & & Y(N-1) 
}
\end{equation*}
where 
$$W = \{(X,V,V'): X \in \End(\C^N), 0 \xrightarrow{1} V \xrightarrow{N-2} V' \xrightarrow{1} \C^N, X\C^N \subset V, XV' = 0\}$$
does {\it not} induce an equivalence. This correspondence is natural since it is isomorphic to the blowup of the zero section of $Y(1) = T^\star \bG(1,N)$ and of $Y(N-1) = T^\star \bG(N-1,N)$. 

From the point of view of categorical $\sl_2$ actions, $\Phi_{\O_W}$ is not an equivalence because $\O_W$ is equal to $\sF^{(N-1)}(0) \in D(Y(1) \times Y(N-1))$ (up to tensoring by a line bundle). This means that the composition $(\O_W)_L * \O_W$ is equal to 
\begin{equation}\label{eq:5}
\sE^{(N-1)}(0) * \sF^{(N-1)}(0) \cong \O_{\Delta} \oplus \sF * \sE \otimes_\C H^\star(\P^{\l-1})
\end{equation}
up to tensoring by a line bundle (this is clearly not equal to $\O_{\Delta}$). Note that relation (\ref{eq:5}) above is a formal consequence of having a categorical $\sl_2$ action (\cite[Lemma 4.2]{CKL3}).

\subsubsection{The case $k=2$.}\label{sec:G(2,4)}
The argument above, namely twisting by the line bundle in (\ref{eq:6}), does {\it not} work here to imply that $\Phi_{\O_{\tZ(2,N)}}$ is an equivalence. Perhaps even more surprising is Namikawa's result \cite{Nam2} that 
$$\Phi_{\O_{Z(2,4)}}: D(Y(2,4)) \rightarrow D(Y(2,4)) \text{ and }
\Phi_{\O_{\tZ(2,4)}}: D(\tY(2,4)) \rightarrow D(\tY(2,4))$$ 
are {\it not} equivalences. 

In \cite{K2} Kawamata tried to tweak the kernel $\O_{\tZ(2,N)}$ to obtain an equivalence. He defined functors $\Psi$ and $\Phi$ as follows. Inside $\tY(2,N)$ there are two natural strata: namely the locus where $X=0$ (isomorphic to $\bG(2,N)$) and the locus where $\rank X \le 1$ (the locus $\rank X \le 2$ is the whole central fibre $Y(2,N)$). Kawamata blows up the first locus and then the strict transform of the second locus to obtain
$$\tY''(2,N) \xrightarrow{f_1} \tY'(2,N) \xrightarrow{f_2} \tY(2,N).$$
Inside $\tY''(2,N)$ we denote by $E_1$ the exceptional divisor of $f_1$ and by $E_2$ the strict transform of the exceptional divisor of $f_2$. Warning: our labeling of divisors does not match precisely that in \cite{K2}. 

Kawamata then blows up $\tY(N-2,N)$ in the same way to obtain $\tY''(N-2,N)$ and identifies this smooth variety with $\tY''(2,N)$. To summarize, we arrive at the following commutative diagram
\begin{equation}\label{eq:kawamata}
\xymatrix{
& \tY''(2,N) \cong \tY''(N-2,N) \ar[dl]_{f} \ar[dr]^{f^+} \ar[d]^{\pi} & \\
\tY(2,N) \ar[rd] & \tZ(2,N) \ar[l]_{\pi_1} \ar[r]^{\pi_2} & \tY(N-2,N) \ar[ld] \\
& \overline{\tB(2,N)}. &
}
\end{equation}
Notice that $\tZ(2,N) = \tY(2,N) \times_{\overline{\tB(2,N)}} \tY(N-2,N)$ so the map $\pi$ exists by the universal property of fibre products. The functors $\Psi$ and $\Phi$ are then defined by
\begin{align}
\Psi(\cdot) &:= f_*(f^{+*}(\cdot) \otimes \O_{\tY''(2,N)}([E_2])) \nonumber \\
\Phi(\cdot) &:= f^+_*(f^{*}(\cdot) \otimes \O_{\tY''(2,N)}((2N-5)[E_2] + (N-3)[E_1])). \nonumber
\end{align}

If we ignore the $\{\cdot\}$ shift for convenience we have:

\begin{proposition}\cite[Prop. 5.7]{C1}\label{prop:main1}
The functor induced by the kernel 
$$\tilde{i}_* \tilde{j}_* \tilde{\sL}(2,N) \otimes \det(\C^N/V) \otimes \det(V')^\vee$$ 
is an isomorphism and together with its adjoint are equal to Kawamata's functors 
$$\Psi, \Phi: D(\tY(2,N)) \rightarrow D(\tY(N-2,N)).$$
\end{proposition}

\subsection{Equivalences and K-theory}

Namikawa's proof that $\Phi_{\O_{\tZ(2,4)}}$ is not an equivalence is via an impressive calculation of $\Hom$ spaces which implies that it is not fully-faithful. It would be interesting to have a more conceptual explanation of this fact. Moreover, it is not known if $\Phi_{\O_{\tZ(k,N)}}$ fails to be an equivalence for any $k \ne 1$ and $N \ge 2k$ (although one expects this is the case).  

Namikawa also shows \cite[Thm. 2.6 and Thm. 2.7]{Nam2} that on K-theory we {\it do} have isomorphisms 
$$[\Phi_{\O_{Z(k,N)}}]: K(Y(\l)) \xrightarrow{\sim} K(Y(-\l)) \text{ and } [\Phi_{\O_{\tZ(k,N)}}]: K(\tY(\l)) \xrightarrow{\sim} K(\tY(-\l))$$
where $K(X)$ denotes the usual Grothendieck group of coherent sheaves on a variety $X$. This fact is a consequence of specialization in K-theory. The argument is as follows. 

For simplicity let us ignore the $\C^\times$-action for a moment. We follow the notation of Chriss and Ginzburg \cite{CG}. Suppose $\tY \rightarrow \A^1$ is a flat family with central fibre $Y$ and denote $\tY^* := \tY \setminus Y$. Then \cite[Sec. 5.3]{CG} they describe a specialization map in K-theory
$$\lim_{t \rightarrow 0}: K(\tY^*) \rightarrow K(Y).$$
More precisely, they show that given any sheaf $\sP^*$ on $\tY^*$ there exists a sheaf $\sP$ on $\tY$ such that :
\begin{itemize}
\item $\sP$ restricts to $\sP^*$ on $\tY^*$
\item $\sP$ has no subsheaves supported on $Y$.
\end{itemize}
Moreover, they show that for any two such sheaves $\sP_1$ and $\sP_2$ their restriction to $K(Y)$ are the same. This restriction is by definition $\lim_{t \rightarrow 0} \sP^*$. 

Now let us apply this to $\tY := \tY(\l) \times_{\A^1} \tY(-\l)$ where $\sP^* = \O_{\tZ(k,N)}|_{\tY^*}$. Then $\sP = \O_{\tZ(k,N)}$ is a possible choice for the extension which means that 
$$\lim_{t \rightarrow 0} \sP^* = [\O_{Z(k,N)}].$$
On the other hand, $\tZ(k,N)$ restricted to $\tY^*$ is actually the graph of an isomorphism $\tY(\l)^* \xrightarrow{\sim} \tY(-\l)^*$ where the $*$ indicates the complement of the central fibre. This immediately implies that $[\sP^*]$ is invertible in K-theory, i.e. $[(\sP^*)_L] * [\sP^*] \cong [\O_{\tilde{\Delta}^*}]$.

Now, in \cite[Thm. 5.3.9]{CG} they also show that the specialization map is compatible with convolution. Since $\lim_{t \rightarrow 0} [\O_{\tilde{\Delta}^*}] = [\O_{\Delta}]$ this implies that $[(\O_{Z(k,N)})_L] * [\O_{Z(k,N)}] \cong [\O_{\Delta}]$ which means that $[\Phi_{\O_{Z(k,N)}}]: K(Y(\l)) \xrightarrow{\sim} K(Y(-\l))$ is invertible. It then follows that $[\Phi_{\O_{\tZ(k,N)}}]$ is also invertible. 

Since the restriction of the line bundle $\tilde{\sL}(k,N)$ from Theorem \ref{thm:main4} to the general fibre is trivial this also means that in K-theory 
$$[\sT(k,N)] = [\O_{Z(k,N)}] \in K(Y(\l) \times Y(-\l)).$$
But when $k=2, N=4$ we know that $\sT(k,N)$ induces an equivalence while $\O_{Z(k,N)}$ does not. In particular, there is no natural specialization map at the level of categories which lifts the one on K-theory. This is unfortunate as it makes it difficult to use deformations to construct and prove derived equivalences. 

\begin{remark} Maulik and Okounkov have been working on a more functorial specialization map in K-theory. It is possible their work can help define a reasonable specialization map at the level of derived categories.
\end{remark}

\section{Geometric categorical $\sl_m$ actions}\label{sec:slm}

One can define the concept of a {\em geometric categorical $\g$ action}\index{geometric categorical ${\mathfrak{g}}$ action} for any simply laced {Kac-Moody Lie algebra}\index{Kac-Moody Lie algebra} $\g$ \cite[Sec. 2.2.2]{CK3}. Let us summarize the definition when $\g = \sl_m$. 

The weight lattice of $\sl_m$ is denoted $X$ and $\fh := X \otimes_\Z \C \cong \A^{m-1}$. We denote the simple and fundamental roots of $\sl_m$ by $\alpha_i, \Lambda_i \in X$ where $i = 1, \dots, m-1$. 

The data of a geometric categorical $\sl_m$ action consists of:
\begin{enumerate}
\item A collection of smooth complex varieties $Y(\l)$ where $\l \in X$ equipped with $\C^\times$-actions. 
\item Fourier-Mukai kernels
\begin{equation*}
\sE^{(r)}_i(\l) \in D(Y(\l) \times Y(\l + r\alpha_i)) \text{ and } \sF^{(r)}_i(\l) \in D(Y(\l + r\alpha_i) \times Y(\l))
\end{equation*}
(which are $\C^\times$ equivariant). 
\item For each $Y(\l)$ a flat deformation $\tY(\l) \rightarrow \fh$ carrying a $\C^\times$-action compatible with the action $x \mapsto t^2x$ (where $t \in \C^\times$) on the base $\fh$. 
\end{enumerate}
\begin{remark} 
Unfortunately, the indexing of $\sE^{(r)}$ and $\sF^{(r)}$ here is slightly different than the convention used in \cite{CKL1,CKL2,CKL3} when $\g = \sl_2$. In that convention, the above $\sE^{(r)}_i(\l)$ should be denoted $\sE^{(r)}_i(\l+\frac{r}{2}\alpha_i)$. We use the notation here because it is more convenient. This is also the convention adopted in \cite{CK3}, \cite{CKL4} and subsequent papers. Sometimes we just write $\sE_i^{(r)}$ and $\sF_i^{(r)}$ when the weight is obvious or irrelevant. 
\end{remark}

On this data we impose the following conditions. 
\begin{enumerate}
\item The spaces $\{Y(\l+r\alpha_i): r \in \Z\}$ together with deformations $\tY(\l+r\alpha_i)$ restricted to $\spn(\alpha_i) \cong \A^1$ and with kernels $\{\sE_i^{(r_1)}(\l+r_2\alpha_i), \sF_i^{(r_1)}(\l+r_2\alpha_i): r_1, r_2 \in \Z\}$ generate a geometric categorical $\sl_2$ action. 
\item If $|i-j|=1$ then 
$$\sE_i * \sE_j * \sE_i \cong \sE_i^{(2)} * \sE_j \oplus \sE_j * \sE_i^{(2)}$$
while if $|i-j|>1$ then $\sE_i * \sE_j \cong \sE_j * \sE_i$.
\item If $i \ne j$ then $\sF_j * \sE_i \cong \sE_i * \sF_j$.
\item The sheaf $\sE_i$ deforms over $\alpha_i^\perp \subset \fh$ to some
$$\tsE_i \in D(\tY(\l)|_{\alpha_i^\perp} \times_{\alpha_i^\perp} \tY(\l+\alpha_i)|_{\alpha_i^\perp}).$$
\item If $|i-j|=1$ then one can show formally from the relations above that there exists a unique non-zero map $T_{ij}: \sE_i * \sE_j [-1] \rightarrow \sE_j * \sE_i$ whose cone we denote
$$\sE_{ij} := \Cone \left( \sE_i * \sE_j [-1] \xrightarrow{T_{ij}} \sE_j * \sE_i \right) \in D(Y(\l) \times Y(\l+\alpha_i+\alpha_j)).$$
Then $\sE_{ij}$ deforms over $B := (\alpha_i+\alpha_j)^\perp \subset \fh$ to some
$$\tsE_{ij} \in D(\tY(\l)|_B \times_{B} \tY(\l+\alpha_i+\alpha_j)|_B).$$
\end{enumerate}

\subsection{Some remarks}

The first condition above summarizes conditions (i)-(vii) in \cite[Sec. 2.2.2]{CK3} while the last four are conditions (viii)-(xi) in \cite{CK3}.

Conditions (ii) and (iii) are just categorical versions of the standard $U_q(\sl_m)$ relations 
$$E_iE_jE_i = \frac{1}{2}\left( E_i^2E_j+E_jE_i^2 \right) \text{ if } |i-j|=1 \text{ and }
E_iE_j = E_jE_i \text{ if } |i-j|>1$$
and $E_iF_j=F_jE_i$ if $i \ne j$. 

To explain the content of (iv) and (v) recall that the deformation $\tY(\l) \rightarrow \fh$ induces a map 
$$\beta_v: \Delta_* \O_{Y(\l)} \rightarrow \Delta_* \O_{Y(\l)} [2]\{-2\}$$
for any $v \in \fh$ by restricting $\tY(\l)$ to $\spn(v)$ and then applying the same construction used to obtain the map in (\ref{eq:7}). The content of (iv) is that for any $v \in \alpha_i^\perp \subset \fh$ both maps 
\begin{align}
\beta_v I &: \Delta_* \O_{Y(\l)} * \sE_i \rightarrow \Delta_* \O_{Y(\l)} * \sE_i [2]\{-2\} \nonumber \\
I \beta_v &: \sE_i * \Delta_* \O_{Y(\l+\alpha_i)} \rightarrow  \sE_i * \Delta_* \O_{Y(\l+\alpha_i)}[2]\{-2\} \nonumber
\end{align}
are equal to zero. 

To see why this is the case consider a deformation $\tY(\l) \rightarrow \A^1$ and denote by $i: Y(\l) \rightarrow \tY(\l)$ the inclusion of the fibre over $0 \in \A^1$. Then for any $\sA \in D(Y(\l))$ there is a natural distinguished triangle
\begin{equation}\label{eq:8}
\sA [1]\{-2\} \rightarrow i^* i_* \sA \rightarrow \sA 
\end{equation}
which induces a connecting map $\alpha: \sA [-1] \rightarrow \sA [1]\{-2\}$. But if $\sA = i^* \tsA$ for some $\tsA \in D(\tY(\l))$ (i.e $\sA$ deforms) then 
$$i^* i_* \sA \cong i^* i_* i^* \tsA \cong i^* (\tsA \otimes i_* \O_{Y(\l)}) \cong i^* \tsA \otimes i^* i_* \O_{Y(\l)} \cong \sA \oplus \sA[1]\{-2\}.$$
This means that $\alpha=0$. On the other hand, one can check that $\alpha$ is the same as the map $\beta I: \Delta_* \O_{Y(\l)} * \sA \rightarrow \Delta_* \O_{Y(\l)} * \sA [2]\{-2\}$ which means that $\beta I=0$. 

In condition (v) the object $\sE_{ij}$ should be thought of as the $\sE$ for the root $\alpha_i + \alpha_j$. Then the content of (v) is analogous to that of (iv), namely, it states that for any $v \in (\alpha_i+\alpha_j)^\perp \subset \fh$ both maps 
\begin{align}
\beta_v I &: \Delta_* \O_{Y(\l)} * \sE_{ij} \rightarrow \Delta_* \O_{Y(\l)} * \sE_{ij} [2]\{-2\} \nonumber \\
I \beta_v &: \sE_{ij} * \Delta_* \O_{Y(\l+\alpha_i+\alpha_j)} \rightarrow  \sE_{ij} * \Delta_* \O_{Y(\l+\alpha_i+\alpha_j)}[2]\{-2\} \nonumber
\end{align}
are zero. 

Condition (v) was included in \cite{CK3} because in practice $\sE_{ij}$ is a natural kernel supported on the union of $\mbox{supp}(\sE_i * \sE_j)$ and $\mbox{supp}(\sE_j * \sE_i)$ and one can write down the deformation $\tsE_{ij}$ fairly easily and explicitly. Moreover, from a geometric standpoint, it is interesting to see these deformations. However, the r\^ole of (v) in proving the braid relation in \cite{CK3} is quite minimal. Namely, it is used to show in \cite[Lemma 4.9]{CK3} that the map
$$\sE_i * \sE_j * \sE_i [-1]\{1\} \xrightarrow{T_{ij}I} \sE_j * \sE_i * \sE_i$$
induces an isomorphism between the summands $\sE_j * \sE_i^{(2)} [-1] \{1\}$ on either side. However, it turns out one can check this directly without the need of (v) (see \cite{C2}). Thus, condition (v) is essentially redundant. 

\subsection{{Braid group actions}\index{braid groups actions}}

First, recall some basic facts about the weight lattice $X$ of $\sl_m$. The weight lattice $X$ comes equipped with a symmetric bilinear pairing $\la \cdot, \cdot \ra$. Under this pairing we have $\la \alpha_i, \Lambda_j \ra  = \delta_{ij}$ and $\la \alpha_i, \alpha_j \ra$ equals $2,-1$ or $0$ depending on whether $i=j$, $|i-j|=1$ or $|i-j|>1$ respectively.

The Weyl group of $\sl_m$ is the symmetric group $S_m$ generated by $s_1, \dots, s_{m-1}$ with the usual relations $s_i^2=1$, $s_is_j=s_js_i$ if $|i-j|>1$ and $s_is_js_i=s_js_is_j$ if $|i-j|=1$. It acts on the weight lattice $X$ via 
$$s_i(\l) := \l - \la \l, \alpha_i \ra \alpha_i.$$

Recall that our motivation behind studying geometric categorical $\sl_2$ actions is that they induce equivalences (Theorem \ref{thm:main1}). A geometric categorical $\sl_m$ action contains $m-1$ different $\sl_2$ actions, generated by $\sE_i$ and $\sF_i$ for $i=1, \dots, m-1$. Thus, for each $i$ one can write down a complex $\Theta_*$ just like the one in (\ref{eq:9}) but where 
$$\Theta_s := \sF_i^{(\la \l, \alpha_i \ra+s)} * \sE_i^{(s)} [-s]\{s\} \in D(Y(\l) \times Y(s_i \cdot \l)).$$ 
Here we are assuming that $\la \l, \alpha_i \ra \ge 0$ (there is a similar complex if $\la \l, \alpha_i \ra \le 0$). 
These complexes have unique right convolutions, denoted $\sT_i(\l) \in D(Y(\l) \times Y(s_i \cdot \l))$, which induce equivalences. In particular, this means that any two varieties  in the same Weyl orbit are derived equivalent.  

But having a geometric categorical $\sl_m$ action is more than having $m-1$ geometric categorical $\sl_2$ actions. This extra structure leads to the following result:

\begin{theorem}\cite[Thm. 2.10]{CK3}\label{thm:main5}
The kernels $\sT_i$ satisfy the braid group relations. Namely, $\sT_i * \sT_j \cong \sT_j * \sT_i$ if $|i-j|>1$ and $\sT_i*\sT_j*\sT_i \cong \sT_j*\sT_i*\sT_j$ if $|i-j|=1$. This gives an action of the braid group $B_m$ on $D(\sqcup Y(\l))$ compatible with the action of the Weyl group on the weight lattice.
\end{theorem}

The key step in the proof of Theorem \ref{thm:main5} lies in proving that if $|i-j|=1$ then $\sE_{ij} * \sT_i \cong \sT_i * \sE_j$ \cite[Cor. 5.4]{CK3}. This implies that $\sT_{ij} * \sT_i \cong \sT_i * \sT_j$ where $\sT_{ij}$ is the equivalence build out of $\sE_{ij}$ and $\sF_{ij}$ (in other words, this is the equivalence induced by the $\sl_2$ action corresponding to the root $\alpha_i + \alpha_j$). It then follows by a similar argument that $\sT_j * \sE_{ij} \cong \sT_i * \sT_j$ which means $\sT_j * \sT_{ij} \cong \sT_i * \sT_j$. Thus 
$$\sT_j^{-1} * \sT_i * \sT_j \cong \sT_{ij} \cong \sT_i * \sT_j * \sT_i^{-1}$$
and we are done. 

\subsection{Examples}
We now describe some examples of geometric categorical $\sl_m$ actions. 

\subsubsection{Example: resolution of {Kleinian singularities}\index{Kleinian singularities}}\label{sec:ex1}

Consider the standard diagonal embedding of $G := \Z/m\Z$ inside $SL_2(\C)$ and let $\pi : Y \rightarrow \C^2/G$ be the minimal resolution. Recall that the exceptional fibre $\pi^{-1}(0)$ contains a chain of $m-1$ $\P^1$s which we label $E_1, \dots, E_{m-1}$. 

We can construct a geometric categorical $\sl_m$ action as follows. We let $Y(0) := Y$, $Y(\l) := \pt$ for $\l$ a root of $\sl_m$ and $Y(\l) := \emptyset$ for all other $\l \in X$. The action of $ \C^\times $ on $Y$ comes from the scaling action on $\C^2$. We define $\sE_i(0): D(Y) \rightarrow D(\pt)$ using the kernel $\O_{E_i}(-1) \in D(Y \times \pt)$ and similarly with $\sE_i(-\alpha_i)$, $\sF_i(0)$ and $\sF_i(-\alpha_i)$ (all other $ \sE_i, \sF_i $ we need to define are functors $D(\pt) \rightarrow D(\pt)$ which we take to be the identity). The deformation $\tilde{Y}$ of $Y$ is the standard deformation which may be constructed by thinking of $\C^2/G $ as a Slodowy slice or by deforming the polynomial defining the singularity $\C^2/G$.

{\bf Fact.} This defines a geometric categorical $\sl_m$ action. 

Let us check condition (ii) for having a geometric categorical $ \sl_m $ action (all other conditions are immediate or follow along the same lines). Suppose $|i-j|=1$ so that $E_i$ and $E_j$ intersect in a point. Then condition (ii) states that
$$ \sE_i(\alpha_{j}) * \sE_{j}(0) * \sE_i (-\alpha_i) \cong \sE_{j}(\alpha_i) * \sE^{(2)}_i(-\alpha_i) \oplus \sE_i^{(2)}(-\alpha_i + \alpha_{j}) * \sE_{j}(-\alpha_i).$$
Now $ Y(-\alpha_i + \alpha_{j}) = \emptyset$ while $\sE^{(2)}_i(-\alpha_i) = \O_{\Delta_{\pt}} = \sE_i(\alpha_{j})$. So we see that this is equivalent to the fact that the composition
$$D(\pt) \xrightarrow{\sE_i(-\alpha_i)} D(Y) \xrightarrow{\sE_{j}(0)} D(\pt)$$
is the identity. Since the first functor is given by tensoring with $\O_{E_i}(-1)$ and the second functor is $\Ext^*(\O_{E_j}(-1)[-1], \cdot)$ this condition corresponds to the fact that $\Ext^\ell(\O_{E_j}(-1), \O_{E_i}(-1))$ is zero unless $\ell=1$ in which case it is one-dimensional.

\subsubsection{Example: {flag varieties}\index{flag varieties}}\label{sec:ex2}

The following generalizes the geometric categorical $\sl_2$ action on cotangent bundles to Grassmannians from section \ref{sec:action}. 

Fix $m \le N$ and consider the variety $Fl_m(\C^N)$ of $m$-step flags in $\C^N$. This variety has many connected components, which are indexed by the possible dimensions of the spaces in the flags. In particular, let
\begin{equation*}
C(m, N) := \{ \l = (\l_1, \dots, \l_m) \in \mathbb{N}^N: \l_1 + \dots + \l_m = N \}.
\end{equation*}
For $ \l \in C(m,N)$, we can consider the variety of $m$-steps flags where the jumps are given by $ \l$:
\begin{equation*}
Fl_\l(\C^N) := \{ 0= V_0 \subset V_1 \subset \cdots \subset V_m = \C^N : \dim V_i/V_{i-1} = \l_i \}.
\end{equation*}

Let $ Y(\l) := T^\star Fl_\l(\C^N)$ (if $\l \not\in C(m,N)$ we take $Y(\l) = \emptyset$).  These will be our varieties for the geometric categorical $\sl_m $ action.  We regard each $ \l $ as a weight for $ \sl_m $ via the identification of the weight lattice of $ \sl_m $ with the quotient $ \Z^m / (1, \cdots, 1)$.  By convention the simple root $\alpha_i$ equals $(0, \dots, 0,-1,1,0, \dots, 0)$ where the $-1$ is in position $i$.

We will use the following description of the cotangent bundle to the partial flag varieties (this generalizes the description for Grassmannians in (\ref{eq:Grassdef})):
\begin{equation*}
Y(\l) := \{(X, V): X \in \End(\C^N), V \in Fl_\l(\C^N), X V_i \subset V_{i-1} \}
\end{equation*}
This description suggests the following deformation $\tilde{Y}(\l) \rightarrow \C^m$ of $Y(\l)$:
\begin{equation*}
\{(X, V, x) : X \in \End(\C^N), V \in Fl_\l(\C^N), x \in \C^m, X V_i \subset V_i, X |_{V_i/V_{i-1}} = x_i \cdot \id \}.
\end{equation*}

We will restrict our deformation over the locus $ \{(x_1, \dots, x_n) \in \C^m : x_m = 0 \} $ which we identify with $ \fh $, the Cartan for $ \sl_m $.

We define an action of $ \C^\times $ on $ \tY(\l) $ by $ t \cdot (X, V, x) = (t^2 X, V, t^2 x) $.  Restricting to $Y(\l) = T^\star Fl_\l(\C^N)$ this corresponds to a trivial action on the base and a scaling of the fibres.

To construct the kernels $\sE^{(r)}_i$ we consider correspondences $W_i^{r}(\l)$ analogous to $W^r(\l)$ defined in section \ref{sec:action} .  More specifically, let $ \l, i, r $ be such that $ \l \in C(m,N) $ and $ \l + r \alpha_i \in C(m,N) $ (i.e. $ \l_i \ge r $).  Then we define 
\begin{align*}
W_i^{r}(\l) := \{ (X, V, V') : & (X, V) \in Y(\l), (X,V') \in Y(\l + r\alpha_i), \\ 
& V_j = V'_j \text{ for } j \ne i, \text{ and } V'_i \subset V_i \}.
\end{align*}
From this correspondence we define 
\begin{equation*}
\sE^{(r)}_i(\l) := \O_{W_i^{r}(\l)} \otimes \det(V_{i+1}/V_i)^{-r} \otimes \det(V'_i/V_{i-1})^r \{r(\l_i-r)\} 
 \end{equation*}
where, abusing notation, $V_i$ denotes the vector bundle on $Y(\l)$ whose fibre over $(X,V) \in Y(\l)$ is $V_i$. This sheaf belongs to $D(Y(\l) \times Y(\l+r\alpha_i))$. Similarly, we define 
\begin{equation*}
\sF^{(r)}_i(\l) := \O_{W_i^{r}(\l)} \otimes \det(V'_i/V_i)^{\l_{i+1} - \l_i + r}\{ r \l_{i+1} \}  \in D(Y(\l+r \alpha_i) \times Y(\l)).
 \end{equation*}
Note that now we regard $W_i^{r}(\l)$ as a subvariety of $ Y(\l+r \alpha_i) \times Y(\l)$ which means that $V_i \subset V_i'$ (since, by convention, the prime indicates pullback from the second factor).

\begin{theorem}\cite[Thm. 3.1]{CK3}\label{thm:main6}
The datum above defines a geometric categorical $\sl_m$ action on $D(T^\star Fl_m(\C^N))$.
\end{theorem}

As a corollary of Theorems \ref{thm:main5} and \ref{thm:main6} we obtain:

\begin{corollary}\cite[Thm. 3.2]{CK3}\label{cor:main2}
There is an action of the braid group $B_m$ on the derived category of coherent sheaves on $T^\star Fl_m(\C^N)$ which is compatible with the action of $S_m$ on the set of connected components $C(m,N)$.
\end{corollary}

Although the construction of each kernel
$$\sT_i \in D(T^\star Fl_\l(\C^N) \times T^\star Fl_{s_i \cdot \l}(\C^N))$$
via a categorical $\sl_m$ action is elaborate, $\sT_i$ has a concrete description like the one in Theorem \ref{thm:main3} (which is just the special case $m=2$). In particular, $\sT_i$ is a Cohen-Macaulay sheaf supported on the variety
\begin{align}
Z_i(\l) := \{(X,V,V')&: X \in \End(\C^N), V \in Fl_\l(\C^N), V' \in Fl_{s_i \cdot \l}(\C^N) \nonumber \\
& XV_j \subset V_{j-1}, XV'_j \subset V'_{j-1} \text{ and } V_j=V'_j \text{ if } j \ne i \}. \nonumber
\end{align}

\subsubsection{Special cases of interest.}
If $N = dm$ for some integer $d$ and we choose $\l = (d, \dots, d)$ then $s_i \cdot \l = \l$ for all $i$. Thus we obtain an action of the braid group on $D(T^\star Fl_\l(\C^N))$ which is a connected variety.

Furthermore, if $d=1$ (i.e. $N=m$) then $T^\star Fl_\l(\C^N)$ is the cotangent bundle of the full flag variety of $\C^N$. An action of the braid group on the full flag variety was also constructed by Khovanov-Thomas \cite{KT} and by Riche \cite{Ric}, Bezrukavnikov-Riche \cite{BR}. In this case, the generators of the braid group act by Seidel-Thomas (a.k.a. spherical) twists (see section \ref{sec:spherical} below). Their work served as motivation for constructing the braid group actions above between more general partial flag varieties.  

\subsubsection{Affine braid groups.}
The (extended) {\em affine braid group}\index{braid groups} of $\sl_m$ has generators $T_i$ and $\theta_i$ for $i=1, \dots, m-1$ and relations:
\begin{itemize}
\item $T_iT_j = T_jT_i$ if $|i-j|>1$ and $T_iT_jT_i=T_jT_iT_j$ if $|i-j|=1$
\item $T_i \theta_j = \theta_j T_i$ if $i \ne j$ 
\item $T_i = \theta_{i-1}^{-1} \theta_{i+1}^{-1} \theta_i T_i^{-1} \theta_i$ for all $i$
\item $\theta_i \theta_j = \theta_j \theta_i$ for all $i,j$.
\end{itemize}

In \cite{CKL4} we show that the braid group action on $D(T^\star Fl_m(\C^N))$ extendes to an affine braid group action as follows:

\begin{corollary}\cite[Thm. 7.2]{CKL4}\label{cor:main3}
The kernels $\sT_i$ together with $\theta_i := \Delta_* \det(V_i)$ induce an action of the (extended) affine braid group of $\sl_m$ on the {\emph{non-equivariant}} derived category of coherent sheaves on $T^\star Fl_m(\C^N)$. 
\end{corollary}

In the above Corollary $\theta_i$ is the kernel inducing the functor which is tensoring with the line bundle $\det(V_i)$. Note that this extended action is only possible on the non-equivariant categories as, perhaps a bit surprisingly, the equivariant shifts $\{\cdot\}$ do not work out. Note that the majority of the content in Corollary \ref{cor:main3} is captured in the relation $T_i = \theta_{i-1}^{-1} \theta_{i+1}^{-1} \theta_i T_i^{-1} \theta_i$ which generalizes the result in Proposition \ref{prop:main2}. 

Again, when $m=N$, this affine braid group action on the full flag variety was constructed first in \cite{KT,Ric,BR}.  

\subsubsection{Example: quiver varieties}\label{sec:ex3}

The varieties in the two examples above (sections \ref{sec:ex1} and \ref{sec:ex2}) are special cases of {\em Nakajima quiver varieties}\index{quiver varieties} of type $A_{m-1}$ \cite{Nak1, Nak2}. In \cite[Thm. 3.2]{CKL4} we construct a geometric categorical $\sl_m$ action on derived categories of coherent sheaves on arbitrary Nakajima quiver varieties of type $A_{m-1}$ (in fact we do this for any simply laced Kac-Moody Lie algebra $\g$). This action recovers the two examples above as special cases.  

Note that the geometry involved in constructing the geometric categorical actions on arbitrary Nakajima quiver varieties is precisely the geometry of cotangent bundles on Grassmannians from section \ref{sec:flops1}. In particular, the generators of the braid group actions induce derived equivalences between varieties, such as $T^\star Fl_\l(\C^N)$ and $T^\star Fl_{s_i \cdot \l}(\C^N)$, which are related by stratified Mukai flops. In fact, inspired by work of Nakajima, many of the proofs in \cite{CKL4} reduce the problem to the case of cotangent bundles to Grassmannians. 

These quiver varieties are also equipped with natural deformations. These deformations are related to each other via the geometry of stratified Atiyah flops from section \ref{sec:flops2}.

\section{Twists}

One of the first techniques for constructing derived autoequivalences was that of spherical twists as defined by Seidel and Thomas in \cite{ST}. This notion was generalized by various authors (Horja \cite{Ho}, Anno \cite{An}, and Rouquier \cite{Rold}) to twists in spherical functors (a relative version). Spherical objects were also generalized to $\P$-objects by Huybrechts and Thomas in \cite{HT}. We briefly discuss their work here. 

\subsection{Seidel-Thomas (spherical) twists}\label{sec:spherical}

First recall the definition of a {\em spherical functor}\index{spherical twist}. Let $X, Y$ be varieties (for convenience, we ignore the $\C^\times$-action in this section). Then a FM kernel $\sP \in D(X \times Y)$ is spherical if:
\begin{itemize}
\item $\sP_R \cong \sP_L [k]$ for some $k$.
\item $\Delta_* \O_X \rightarrow \sP_R * \sP \cong \sP_L * \sP [k] \rightarrow \Delta_* \O_X [k]$ is a distinguished triangle in $D(X \times X)$. Both maps here are the adjunction maps. 
\end{itemize}
The induced map $\Phi_{\sP}: D(X) \rightarrow D(Y)$ is called a spherical functor. Define 
$$\sT_{\sP} := \Cone( \sP * \sP_R \xrightarrow{adj} \Delta_* \O_Y) \in D(Y \times Y)$$
where $adj$ is the natural adjunction map. The induced functor $\Phi_{\sT_\sP}: D(Y) \rightarrow D(Y)$ is called a spherical twist. 

\begin{remark}
The second condition above is sometimes replaced by $\sP_R * \sP \cong \Delta_* \O_X \oplus \Delta_* \O_X [k]$ (which is {\it a priori} stronger). The right hand side then resembles the cohomology of a sphere. There is also a mirror side to this story where the twist $\sT_{\sP}$ is often induced by monodromy around a singularity whose vanishing cycle is a sphere. This explains the terminology ``spherical functor''.
\end{remark}

\begin{theorem}\cite{ST,Ho,An,Rold}\label{thm:twist} If $\sP \in D(X \times Y)$ is a spherical kernel then $\Phi_{\sT_\sP}: D(Y) \rightarrow D(Y)$ is a derived autoequivalence. 
\end{theorem}

If $X$ is just a point then $\sP \in D(Y)$ is refered to as a spherical object. In this case the setup above recovers the construction from \cite{ST}. 

On the other hand if $k=2$ then a spherical functor is just a special case of a geometric categorical $\sl_2$ action. To see this we take $Y(\l) = \emptyset$ if $\l \not\in \{-2,0,2\}$ while 
$$Y(-2) := X, Y(0) := Y \text{ and }  Y(2) := X.$$ 
We then define 
\begin{align}
\sE(-1) := \sP \in D(X \times Y) \hspace{1cm} & \sE(1) := \sP_R [-1] \in D(Y \times X) \nonumber \\
\sF(1) := \sP \in D(X \times Y) \hspace{1cm}  & \sF(-1) := \sP_R [-1] \in D(Y \times X). \nonumber
\end{align}
It turns out this case is simple enough that we do not need the deformations $\tY(\l)$. Now one can easily check that the geometric categorical relations on the $\sE$s and $\sF$s defined above are equivalent to the fact that $\sP$ is a spherical functor. Furthermore, the complex $\Theta_*$ from section \ref{sec:equivalences} becomes 
$$\left[ \sF * \sE [-1] \rightarrow \Delta_* \O_Y \right] \cong \left[ \sP * \sP_R \rightarrow \Delta_* \O_Y \right]$$
which means that $\sT \cong \sT_{\sP}$ and $\Phi_{\sT_\sP}$ is an autoequivalence by Theorem \ref{thm:main1}. 

This explains how (geometric) categorical $\sl_2$ action induce spherical twists when $k=2$. Now recall that $\sl_2$ is the Lie algebra defined by the Cartan datum consisting of the $1 \times 1$ matrix $(2)$. The matrix $(k)$ for $k > 2$ also defines a (generalized) Lie algebra. Geometric categorical actions of this Lie algebra explain the existence of spherical twists for any $k \ge 2$. In practice these algebras are similar to $\sl_2$ and the equivalences they induce are also very similar (just like spherical twists for $k=2$ and for $k > 2$ are defined by the same basic construction).

\subsection{$\P^n$-twists}

In \cite{HT} Huybrechts and Thomas define $\sP \in D(Y)$ to be a $\P^n$-object if $\sP \cong \sP \otimes \omega_Y$ and $\Ext^*(\sP,\sP) \cong H^*(\P^n,\C)$ as a graded ring. They then prove:

\begin{proposition}\cite[Lemma 2.1, Prop. 2.6]{HT}\label{prop:HT} Let $\sP \in D(Y)$ be a $\P^n$ object and $h \in \Ext^2(\sP,\sP)$ a generator. Then inside $D(Y \times Y)$ the complex 
$$(\sP^\vee \boxtimes \sP) [-2] \xrightarrow{h^\vee \boxtimes \id - \id \boxtimes h} \sP^\vee \boxtimes \sP \xrightarrow{{\rm tr}} \Delta_* \O_Y$$
has a unique right convolution $\sT_{\sP} \in D(Y \times Y)$ which induces an autoequivalence $\Phi_{\sT_{\sP}}: D(Y) \rightarrow D(Y)$ called a $\P^n$-twist.
\end{proposition}

\begin{remark} This is an analogue of a spherical object with $k=2$. One can clearly replace $H^*(\P^n,\C)$ above with a ring where the degree jumps are some arbitrary $k \ge 2$. Everything in this section works in this greater generality but for exposition purposes we restrict to the case $k=2$. 
\end{remark}

In analogy with spherical functors one can define a {$\P^n$ functor}\index{${\mathbb{P}}^n$ functor} as follows. A FM kernel $\sP \in D(X \times Y)$ is a $\P^n$ kernel if:
\begin{itemize}
\item $\sP_R \cong \sP_L [2n]$.
\item $\sH^* (\sP_R * \sP) \cong \Delta_* \O_X \otimes_\C H^*(\P^n,\C)$ where $\sH^*(\cdot)$ denotes the cohomology sheaves
\item there exists a map $\beta: \Delta_* \O_Y \rightarrow \Delta_* \O_Y [2]$ in $D(Y \times Y)$ so that 
$$I \beta I: \sP_R * (\Delta_* \O_Y) * \sP \rightarrow \sP_R * (\Delta_* \O_Y) * \sP [2]$$
induces an isomorphism (at the level of cohomology) between $n$ summands $\Delta_* \O_X$ on either side. 
\end{itemize}

Note that the third condition is the analogue of the fact that $\Ext^*(\sP,\sP)$ is isomorphic to $H^*(\P^n,\C)$ as a ring (rather than as a vector space). One could replace this condition with a ring condition on $\sP_R * \sP$ but the language above seems more convenient. 

\begin{remark} 
Nick Addington recently gave a similar definition of a $\P^n$ functor in \cite[Section 3]{Ad}. Although his definition is slightly more general the key ideas and properties are the same. 
\end{remark}

Then the analogue of Proposition \ref{prop:HT} is:

\begin{proposition}\label{prop:conv}
Let $\sP \in D(X \times Y)$ be a $\P^n$ functor and suppose $HH^1(X) = 0$ (where $HH^*$ denotes Hochschild cohomology). Then inside $D(Y \times Y)$ there is a complex 
\begin{equation}\label{eq:20}
(\sP * \sP_R)[-2] \xrightarrow{\beta I * I - I * I \beta} (\sP * \sP_R) \xrightarrow{adj} \Delta_* \O_Y
\end{equation}
that has a unique right convolution $\sT_{\sP} \in D(Y \times Y)$. This kernel induces an autoequivalence $\Phi_{\sT_{\sP}}: D(Y) \rightarrow D(Y)$. 
\end{proposition}
\begin{proof}
The proof of Proposition \ref{prop:HT} given in \cite{HT} generalizes directly. The only tricky point is to show that the convolution is unique. Using section \ref{sec:convcpx} it suffices to check that $\Hom((\sP * \sP_R)[-2], \Delta_* \O_Y) = 0$. Now 
\begin{align}
\Hom((\sP * \sP_R)[-2], \Delta_* \O_Y [-1]) 
&\cong \Hom(\sP, \sP [1]) \nonumber \\
&\cong \Hom(\sP_L * \sP, \Delta_* \O_X [1]) \nonumber \\
&\cong \Hom(\Delta_* \O_X \oplus \Delta_* \O_X [2], \Delta_* \O_X [1]) \nonumber \\
&\cong HH^1(X) \oplus HH^{-1}(X). \nonumber
\end{align}
Since $HH^{-1}(X)=0$ the result follows. 
\end{proof}
\begin{remark} 
I would like to thank Nick Addington for pointing out that one should add the condition $HH^1(X)=0$ as part of the hypothesis in Proposition \ref{prop:conv}. However, I suspect that the convolution of (\ref{eq:20}) is unique even if $HH^1(X) \ne 0$ (this condition is sufficient but not necessary). However, the proof involves a lot of diagram chasing so we leave it up to the reader as an exercise/conjecture. 
\end{remark}

It turns out $\P^n$ functors are also closely related to categorical $\sl_2$ actions. To see this consider a geometric categorical $\sl_2$ action where $Y(\l) = \emptyset$ for $\l > n+1$. Denote $Y := Y(n-1)$ and $X := Y(n+1)$ and let $\sP := \sF(n-1) \in D(X \times Y)$. Furthermore, let $\beta: \Delta_* \O_Y \rightarrow \Delta_* \O_Y [2]$ be the map defined using the deformations $\tY(\l)$ of $Y(\l)$ as in section \ref{sec:remarks}. 

We claim that $\sP = \sF(n-1)$ is a $\P^n$-kernel. The first two conditions are easy consequences of conditions (iii) and (v) in section \ref{sec:def}. The last condition is harder to see but essentially follows from condition (vi). 

Now, we have 
$$\sT(n-1) \cong \Cone(\sF^{(n)} * \sE [-1] \rightarrow \sF^{(n-1)}) \in D(Y(n-1) \times Y(-n+1))$$
since $\sF^{(n+s)} * \sE^{(s)} = 0$ for $s > 1$. And likewise 
$$\sT(-n+1) \cong \Cone(\sE^{(n)} * \sF [-1] \rightarrow \sE^{(n-1)}) \in D(Y(-n+1) \times Y(n-1)).$$
Thus $\sT(-n+1) * \sT(n-1) \in D(Y(n-1) \times Y(n-1))$ is given by the right convolution of a complex:
$$\sE^{(n)} * \sF * \sF^{(n)} * \sE [-2] \rightarrow 
\begin{matrix} \sE^{(n)} * \sF * \sF^{(n-1)} \\ \oplus \sE^{(n-1)} * \sF^{(n)} * \sE [-1] \end{matrix} \rightarrow \sE^{(n-1)} * \sF^{(n-1)}$$
which simplifies to 
$$\sF * \sE \otimes_\C H^\star(\P^n) [-2] \rightarrow 
\begin{matrix} \sF * \sE \otimes_\C H^\star(\P^{n-1}) [-1] \\
\oplus \sF * \sE \otimes_\C H^\star(\P^{n-1}) [-1] \end{matrix}
\rightarrow \id \oplus \sF * \sE \otimes H^\star(\P^{n-2})$$
where we use the convention for $H^\star(\P^n)$ from section \ref{sec:def} (i.e. symmetric with respect to degree zero). Note that this simplification uses some basic commutation relations between $\sE$s and $\sF$s (Lemma 4.2 of \cite{CKL3}) which follow formally from the relations in section \ref{sec:def}.  

It is not too hard to check that the second map is surjective on summands of the form $\sF*\sE$ while the first map is injective on $n$ such summands. It follows that this complex is homotopic to one of the form 
$$\sF * \sE [-n-2] \rightarrow \sF * \sE [-n] \rightarrow \id.$$
The second map in this complex is unique and hence must be the adjunction map (up to a non-zero multiple). The first map is a little harder to deduce but it turns out to be equal to $\beta I * I - I * I \beta$ where $\beta$ is the map in equation (\ref{eq:7}). We conclude that:

\begin{proposition}\label{prop:main3} Given a geometric categorical $\sl_2$ action with $Y(\l)=\emptyset$ for $\l > n+1$ it follows that $\sT(-n+1) * \sT(n-1) \in D(Y(n-1) \times Y(n-1))$ is isomorphic to the unique right convolution of
\begin{equation}\label{eq:14}
\sF * \sE [-n-2] \xrightarrow{\beta I * I - I * I \beta} \sF * \sE [-n] \xrightarrow{adj} \id.
\end{equation}
Moreover, the $\sl_2$ action induces a $\P^n$ kernel $\sP := \sF(n-1) \in D(X \times Y)$ where $Y = Y(n-1)$ and $X = Y(n+1)$ such that the induced $\P^n$ twist $\sT_\sP$ is isomorphic to $\sT(-n+1) * \sT(n-1)$. 
\end{proposition}

\begin{remark}
Notice that unlike the spherical functor case, a $\P^n$ functor does not induce a geometric categorical $\sl_2$ action because it does not give us the spaces $Y(\l)$ for $\l \ne n-1,n+1$. 
\end{remark}

One can imagine trying to compute $\sT(-n+1) * \sT(n-1)$ even if $Y(\l)$ is not empty for $\l > n+1$. Indeed, one can probably obtain some reasonable expressions for these kernels in terms on complexes where all the terms are of the form $\sF^{(s)} * \sE^{(s)}$ for some $s$. 

One could also try to define a Grassmannian $\bG(k,N)$ object $\sP$ (generalizing $\P^n$ objects). The data for this should contain spaces $Y, X_1, \dots, X_k$ together with kernels $\sP_i \in D(X_i \times Y)$ for $i=1, \dots, k$. This is because, imitating the construction above, one needs spaces 
$$Y:= Y(n-1), X_1 := Y(n+1), \dots, X_k := Y(n-1+2k)$$ 
together with kernels $\sF^{(i)} \in D(Y(n-1+2i) \times Y(n-1))$ for $i=1,\dots,k$. 

\subsection{Infinite twists and some geometry}

Proposition \ref{prop:main3} when $n=1$ states that, in a geometric categorical $\sl_2$ action, if $Y(\l) = \emptyset$ for $\l > 2$ then $\sT^2 \in D(Y(0) \times Y(0))$ is the right convolution of the complex (\ref{eq:14}). Hence $\sT^{-2}$ is given by the adjoint which is the left convolution of the complex 
$$\id \xrightarrow{adj} \sF * \sE [1] \xrightarrow{\beta I * I - I * I \beta} \sF * \sE [3].$$
More generally, if you look at $\sT^{-2 \ell} \in D(Y(0) \times Y(0))$ then a little bit of work shows that it is given as the unique left convolution of
$$\id \xrightarrow{adj} \sF * \sE [1] \xrightarrow{\beta I * I - I * I \beta} \dots \xrightarrow{\beta I * I + I * I \beta} \sF * \sE [2 \ell -1] \xrightarrow{\beta I * I - I * I \beta} \sF * \sE [2 \ell +1].$$
If we let $\ell \rightarrow \infty$ then this complex converges to 
\begin{equation}\label{eq:15}
\id \xrightarrow{adj} \sF * \sE [1] \rightarrow \sF * \sE [3] \rightarrow \dots \rightarrow \sF * \sE [2 \ell -1] \rightarrow \sF * \sE [2 \ell +1] \rightarrow \dots
\end{equation}
where the differentials after the left hand adjunction map alternate between $(\beta I * I - I * I \beta)$ and $(\beta I * I - I * I \beta)$. Here we say that a sequence of complexes converges if it eventually stabilizes in any given degree (see, for instance, \cite[Sec. 3]{Roz} for more details). We denote the left convolution of (\ref{eq:15}) by $\sT^{-\infty}$. 

The object $\sT^{-\infty}$ lives naturally in $D^-(Y(0) \times Y(0))$ which is the bounded above derived category of coherent sheaves. This might seem strange since the complex (\ref{eq:15}) is bounded below. However, $\sF * \sE$ is some bounded complex and $\sF * \sE [2 \ell -1]$, when you perform the left convolution, is shifted by $[2 \ell -1- \ell] = [\ell-1]$ so as $\ell \rightarrow \infty$ this is shifted lower and lower in cohomology which explains why it belongs to $D^-$ and not $D^+$. 

Now consider the geometric categorical $\sl_2$ action on $Y(0) := T^\star \P^1$ where $Y(2) = Y(-2)$ are points and $\sE \in D(Y(0) \times \mbox{pt})$ is given by the twisted zero section $\O_{\P^1}(-1)$ and the same with $\sF \in D(\mbox{pt} \times Y(0))$. The map $\beta: \sE \rightarrow \sE [2]$ is the unique map in $\Ext^2_{T^\star \P^1}(\O_{\P^1}(-1), \O_{\P^1}(-1))$. Now
$$\sF * \sE \cong \O_{\P_1}(-1) \boxtimes \O_{\P_1}(-1) \in D(T^\star \P^1 \times T^\star \P^1).$$
This means that $\sT^{-\infty} \in D^-(T^\star \P^1 \times T^\star \P^1)$ is a complex whose cohomology $\sH^i(\sT^{-\infty})$ isomorphic to 
\begin{itemize}
\item $\O_{\P_1}(-1) \boxtimes \O_{\P_1}(-1)$ if $i < 0$ 
\item is a sheaf supported on $\P^1 \times \P^1 \cup \Delta \subset T^\star \P^1 \times T^\star \P^1$ if $i=0$ 
\item $0$ if $i > 0$.
\end{itemize}

Recall the map $p(1,2): T^\star \P^1 \rightarrow \overline{B(1,2)}$. In this case $\overline{B(1,2)}$ is just the quadric $Q \subset \C^3$ and $p(1,2)$ (or $p$ for short) is the map which collapses the zero section inside $T^\star \P^1$ to a point. The composition $p^* p_*$ does not preserve the bounded derived category since $Q$ is singular but it does preserve the bounded above derived category. 

\begin{proposition}\label{prop:main4}
The composition $p^* p_*: D^-(T^\star \P^1) \rightarrow D^-(T^\star \P^1)$ is induced by the kernel $\sT^{-\infty} \in D^-(T^\star \P^1 \times T^\star \P^1)$. 
\end{proposition}
\begin{proof}
The pushforward and pullback maps from $T^\star \P^1$ to $Q$ are given by the graph $\Gamma_p$ of $p$ which is the kernel $\O_{\Gamma_p} \in D(T^\star \P^1 \times Q)$. We denote by $\sK$ the convolution $\O_{\Gamma_p} * \O_{\Gamma_p} \in D^-(T^\star \P^1 \times T^\star \P^1)$ which induces $p^* p_*$. 

The adjoint map $p^* p_* (\cdot) \rightarrow (\cdot)$ corresponds to a natural map $\rho: \sK \rightarrow \O_{\Delta}$. Now, since $p_* \O_{\P^1}(-1) = 0$ it follows that 
$$\sK * \sT^{-1} \cong \Cone(\sK * \id \xrightarrow{I adj} \sK * \sF * \sE) \cong \sK.$$
So applying this to $\rho$ we get a map $\sK \rightarrow \sT^{-1}$. Repeating this and taking the limit we obtain a morphism $\hat{\rho}: \sK \rightarrow \sT^{-\infty}$. We would like to show that $\Cone(\hat{\rho})=0$. To do this we show that it acts by zero on any object in $D^-(T^\star \P^1)$. 

Let $\sM \in D^-(T^\star \P^1)$ and consider the exact triangle
$$p^* p_* \sM \xrightarrow{adj} \sM \rightarrow \Cone(adj).$$
Notice that $p_* \Cone(adj) = 0$ so it suffices to show that $\Phi_{\Cone(\hat{\rho})} (\sN)$ for any $\sN$ where either $\sN = p^* \sN'$ for some $\sN' \in D^-(Q)$ or $p_* \sN = 0$. 

If $\sN = p^* \sN'$ then
\begin{equation}\label{eq:17}
\Phi_{\Cone(\hat{\rho})} (p^* \sN') \cong \Cone(\Phi_{\sK} (p^* \sN') \rightarrow \Phi_{\sT^{-\infty}} (\sN')).
\end{equation}
Now $\Phi_{\sK} (p^* \sN') \cong p^* p_* p^* \sN' \cong p^* \sN'$ and $\Phi_{\sT^{-\infty}} (p^* \sN') \cong p^* \sN'$ since, by a straight-forward calculation, $\Phi_{\sE} (p^* \sN') = 0$. It is not hard to see that the map in (\ref{eq:17}) above induces an isomorphism and hence $\Phi_{\Cone(\hat{\rho})} (p^* \sN') = 0$.

On the other hand, suppose $p_* \sN = 0$ then $\sN$. Since the fibres of $p$ are at most one-dimensional this means $p_* \sH^i(\sN) = 0$ for any $i$ and so we can assume $\sN$ is a sheaf. But then by Lemma \ref{lem:1} below $\sN$ is a direct sum of $\O_{\P^1}(-1)$. So it suffices to show that $\Phi_{\Cone(\hat{\rho})}(\O_{\P^1}(-1))=0$. 

To see this we check that $\Phi_{\sK}(\O_{\P^1}(-1))=0$ and $\Phi_{\sT^{-\infty}}(\O_{\P^1}(-1))=0$. The first follows since $p_* \O_{\P^1}(-1)=0$. On the other hand, it is a standard exercise to check that $\Phi_{\sT^{-1}}(\O_{\P^1}(-1)) \cong \O_{\P^1}(-1)[1]$. This means that 
$$\Phi_{\sT^{-2\ell}}(\O_{\P^1}(-1)) \cong \O_{\P^1}(-1))[2\ell]$$ 
and hence, $\Phi_{\sT^{-\infty}}(\O_{\P^1}(-1)) = 0$. This completes the proof. 
\end{proof}

\begin{lemma}\label{lem:1}
If $\sM$ is a coherent sheaf on $T^\star \P^1$ and $p_* \sM = 0$ then $\sM \cong \O_{\P^1}(-1)^{\oplus \ell}$. 
\end{lemma}
\begin{proof}
Suppose $p_* \sM = 0$. Then $\sM$ is set theoretically supported on $\P^1$ because $p$ is an isomorphism away from $\P^1$. Now any sheaf of $\P^1$ is a direct sum of structure sheaf and line bundles. So if $\sM$ were scheme theoretically supported on $\P^1 \subset T^\star \P^1$ then the result would follow because $\O_{\P^1}(-1)$ is the only sheaf with vanishing cohomology. 

More generally, consider the short exact sequence 
\begin{equation}\label{eq:16}
0 \rightarrow \sM'' \rightarrow \sM \rightarrow \sM' \rightarrow 0
\end{equation}
where $\sM'$ is the quotient of $\sM$ scheme theoretically supported on $\P^1$ (equivalently, $\sM''$ is the part of $\sM$ killed by $f$ where $f$ is the local equation defining $\P^1 \subset T^\star \P^1$). It would suffice to show $p_* \sM' = 0$ since then $p_* \sM'' = 0$ and we can proceed up induction to conclude that $\sM''$ and $\sM'$ are both direct sums of $\O_{\P^1}(-1)$ and then the same must be true of $\sM$ since $\Ext^1(\O_{\P^1}(-1), \O_{\P^1}(-1)) = 0$.

We now prove that $p_* \sM' = 0$. From the long exact sequence induced by (\ref{eq:16}) and the fact that $Q$ is affine it suffices to show that $H^0(\sM') = 0$. On the other hand, take the standard short exact sequence 
$$0 \rightarrow \O_{T^\star \P^1}(-\P^1) \rightarrow \O_{T^\star \P^1} \rightarrow \O_{\P^1} \rightarrow 0$$ 
and tensor it with $\sM$. Then $\sH^0(\sM \otimes \O_{\P^1}) \cong \sM'$ so from the long exact sequence it suffices to show that $H^1(\O_{T^\star \P^1}(-\P^1) \otimes \sM) = 0$. But $\O_{T^\star \P^1}(-\P^1) \cong \pi^* \O_{\P^1}(2)$ where $\pi: T^\star \P^1 \rightarrow \P^1$ is the standard projection. Then, by the projection formula
$$\pi_* (\O_{T^\star \P^1}(-\P^1) \otimes \sM) \cong \pi_* (\pi^* \O_{\P^1}(2) \otimes \sM) \cong \O_{\P^1}(2) \otimes \pi_*(\sM).$$
Now $\pi_* \sM$ on $\P^1$ has no cohomology so it must be isomorphic to $\O_{\P^1}(-1)^{\oplus \ell}$ for some $\ell$. Subsequently $\O_{\P^1}(2) \otimes \pi_*(\sM) \cong \O_{\P^1}(1)^{\oplus \ell}$ has no higher cohomology and hence $H^1(\O_{T^\star \P^1}(-\P^1) \otimes \sM) = 0$.
\end{proof}

\begin{remark} There are two points worth noting here. First is that the limit $\sT^{-\infty}$ is a projector (i.e. $\sT^{-\infty} * \sT^{-\infty} \cong \sT^{-\infty}$) even though $\sT^{-\ell}$ is invertible for any $\ell$. Secondly, the kernel $\sT^{-\infty}$, which is defined using a formal construction not involving $p$, has a geometric description as the kernel inducing $p^* p_*$. 
\end{remark}

More generally, one can consider the cotangent bundle to the full flag variety $T^\star Fl(\C^N)$ where there is a straight-forward generalization of the discussion above (which was the case $N=2$). From section \ref{sec:ex2}, one can construct a categorical $\sl_N$ action so that $D(T^\star Fl(\C^N))$ corresponds to the zero weight space. This induces, a braid group action on $D(T^\star Fl(\C^N))$ generated by kernels $\sT^{-1}_i := \Cone(\id \rightarrow \sF_i * \sE_i)$ for $i=1, \dots, N-1$. 

Then the arguments above can be used to show that $\sT_i^{-\infty} := \lim_{\ell \rightarrow \infty} \sT_i^{-2\ell}$ is well defined. Moreover, $\sT_i^{-\infty}$ is isomorphic to the kernel which induces $p_i^* p_{i*}$ where $p_i$ is the projection from $T^\star Fl(\C^N)$ given by forgetting $V_i$. 

\subsubsection{Non-generalizations}

Consider again the situation in Proposition \ref{prop:main3} where $n \ge 0$ is now arbitrary. This time $\sT^{-2\ell} \in D(Y(n-1) \times Y(n-1))$ is isomorphic to the left convolution of the complex 
\begin{align*}
& \id \rightarrow \sF * \sE [n] \rightarrow \sF * \sE [n+2] \rightarrow \sF * \sE [3n+2] \rightarrow \sF * \sE [3n+4] \rightarrow \dots \\
& \dots \rightarrow \sF * \sE [(2\ell-1)(n+1)-1] \rightarrow \sF * \sE [(2\ell-1)(n+1)+1]
\end{align*}
where, after the left-most adjunction map, the maps alternate between $(\beta I * I - I * I \beta)$ and $\sum_{i=0}^n (\beta^i I * I + I * I \beta^{n-i})$. This complex also has an obvious limit as $\ell \rightarrow \infty$ which we denote $\sT^{-\infty} \in D^-(Y(n-1) \times Y(n-1))$. This is completely analogous to the case $n=1$ discussed above. 

Now suppose $Y(\l) = T^\star \bG(k, n+1)$ (where $\l = n+1-2k$) so that $Y(n-1) = T^\star \P^n$. As before we have the map $p(1,n+1): T^\star \P^n \rightarrow \overline{B(1,n+1)}$ which collapses the zero section. However, if $n > 1$, then $\sT^{-\infty} \not\cong \sK$ where $\sK$ is the kernel inducing the map 
$$p(1,n+1)^* p(1,n+1)_*: D^-(T^\star \P^n) \rightarrow D^-(T^\star \P^n).$$ 
 
It turns out $\sK$ is a stronger projection than $\sT^{-\infty}$. In other words,
$$\sK * \sT^{-\infty} \cong \sK \cong \sT^{-\infty} * \sK \in D^-(T^\star \P^n \times T^\star \P^n).$$
The argument used to prove Proposition \ref{prop:main4} fails because the kernel of the map $p(1,n+1)_*$ is now larger (and more complicated) than the kernel of the map $\sT^{-\infty}$. Geometrically, this difference seems to be related to the fact that the singular and intersection cohomologies of $\overline{B(1,n+1)}$ are the same if $n=1$ but different for $n > 1$. More precisely, $\sK$ is akin to singular and $\sT^{-\infty}$ to intersection cohomology. 
 
In \cite{C2} we use $\sT^{-\infty}$ (rather than the geometric kernel $\sK$) to categorify Reshetikhin-Turaev knot invariants. This suggests that $\sT^{-\infty}$ is at least as natural as $\sK$. However, this also begs the obvious question: what is the geometric interpretation of the kernel $\sT^{-\infty} \in D^-(T^\star \P^n \times T^\star \P^n)$ when $n > 1$?

\section{The general flop -- a discussion}

\subsection{The Mukai flop.}
We recall the definition of a general Mukai flop (see for example \cite[Sec. 4]{Nam1}). Let $Y$ be a smooth variety of dimension $2n$ which contains a subvariety $X \subset Y$ isomorphic to $\P^n$ so that $N_{X/Y} \cong T^\star \P^n$. If $Y$ is holomorphic symplectic then this second condition is automatically satisfied. 

Then one can blow up $Y$ and then then blow down to obtain another variety $Y^+$ which also contains a subvariety $X^+ \cong \P^n$ with $N_{X^+/Y^+} \cong T^\star \P^n$. Moreover, there exists maps $Y \rightarrow \overline{Y}$ and $Y^+ \rightarrow \overline{Y}$ to a common (singular) variety $\overline{Y}$ which collapses $X$ and $X^+$ to a point but are isomorphisms away from $X$ and $X^+$. To summarize, we have the following diagram:
\begin{equation*}
\xymatrix{
& Y \times_{\overline{Y}} Y^+ \ar[dl]_{\pi_1} \ar[dr]^{\pi_2} & \\
Y \ar[dr] & & Y \ar[dl] \\
& \overline{Y} &
}
\end{equation*}
Of course, this generalizes our example above where $Y = T^\star \P^n$. As before, $Y \times_{\overline{Y}} Y^+$ is a variety with two equidimensional components. Namikawa proves:

\begin{proposition}\cite[Sec. 4]{Nam1}\label{prop:3} There exists an isomorphism
$$\pi_{2*} \pi_1^*: D(Y) \xrightarrow{\sim} D(Y^+).$$ 
\end{proposition}

Namikawa first checks this isomorphism for the local case of $Y = T^\star \P^n$. He then uses the fact that the formal neighbourhoods of $X \subset Y$ and $\P^n \subset T^\star \P^n$ are isomorphic to prove the more general case above.

\subsection{The stratified Mukai flop of type A}
An abstract definition of a stratified Mukai flop was first discussed by Markman in \cite{M} while studying the geometry of the moduli spaces of sheaves on K3 surfaces. The idea is to immitate the geometry of the situation in diagram (\ref{eq:10}). More specifically, there is a filtration 
$$T^\star \bG(k,N) \supset T^\star \bG(k,N)_1 \supset \dots \supset T^\star \bG(k,N)_k$$
where $ T^\star \bG(k,N)_i$ is the subvariety 
\begin{align*}
\{ (X,V): & X \in \End(\C^N), 0 \xrightarrow{k} V \xrightarrow{N-k} \C^N, X \C^N \subset V \\
& \text{ and } X V \subset 0, \dim (\ker X) \ge N-k+i \}
\end{align*}
and the projection map $p(k,N): T^\star \bG(k,N) \rightarrow \overline{B(k,N)}$ restricted to $T^\star \bG(k,N)_i \setminus T^\star \bG(k,N)_{i+1}$ is a $\bG(i,N-k+i)$-fibration. In this paper we will use the definition from \cite{FW}, which is similar to that in \cite{M} but fits better for our discussion. 

Let $Y$ and $Y^+$ be two smooth varieties equipped with two collections of closed subvarieties 
$$X_k \subset \dots \subset X_1 \subset Y \text{ and } X_k^+ \subset \dots \subset X_1^+ \subset Y^+.$$ 
Now assume there are two birational maps $Y \xrightarrow{p} \overline{Y} \xleftarrow{p^+} Y^+$ and denote by $f: Y \dashrightarrow Y^+$ the induced birational map. Then this data describes a stratified Mukai flop of type $A_{N,k}$ (where $2k \le N$) if the following conditions hold.
\begin{itemize}
\item $f$ induces an isomorphism $Y \setminus X_1 \xrightarrow{\sim} Y^+ \setminus X^+_1$. 
\item $p(X_j) = p^+(X_j^+)$ for $j=1, \dots, k$. We denote $S_j := p(X_j)$. 
\item $S_k$ is smooth and $p|_{X_k}: X_k \rightarrow S_k$ is isomorphic to the projection map $\bG(k,W) \rightarrow S_k$ where $W$ is some $N$-dimensional vector bundle on $S_k$ and $\bG(k,W)$ denotes the relative Grassmannian of $k$-planes. Moreover, the normal bundle $N_{X_k/Y}$ is isomorphic to the relative cotangent bundle $T^\star_{X_k/S_k}$. The same thing holds for $p^+|_{X^+_k}: X_k^+ \rightarrow S_k$ with $W$ replaced by $W^\vee$. 
\item If $k=1$ this should be the usual Mukai flop. If $k \ge 2$ let $\Bl_{X_k} Y, \Bl_{X^+_k} Y^+$ and $\Bl_{S_k} \overline{Y}$ denote the blowups of $Y,Y^+$ and $\overline{Y}$ in $X_k, X_k^+$ and $S_k$ respectively. Then the proper transforms of all $X_j$ and $X^+_j$ together with the birational maps $\Bl_{X_k} Y \rightarrow \Bl_{S_k} \overline{Y} \leftarrow \Bl_{X_k^+} Y^+$ must describe a stratified Mukai flop of type $A_{N-2,k-1}$.  
\end{itemize}

\begin{remark}
Given just the $A_{N,k}$ contraction $p: Y \rightarrow \overline{Y}$ it follows (just like for the usual Mukai flop) that the corresponding stratified Mukai flop exists. For a proof see, for instance, Proposition 2.1 of \cite{FW}. 
\end{remark}

Of course, one would like an analogue of Proposition \ref{prop:3} in the spirit of Theorem \ref{thm:main3}. This theorem would identify an open subset inside $Y \times_{\overline{Y}} Y^+$ which is an analogue of $Z^o(k,N)$ and a line bundle on it so that the pushforward of this line bundle is a kernel which induces a derived equivalence $D(Y) \xrightarrow{\sim} D(Y^+)$.

One possible approach to proving this equivalence is to deform to the normal cone. This means looking at $Y \times \A^1$ and blowing up $X_k \times \{0\} \subset Y$ (and likewise with $Y^+$). It is shown in \cite[Sec. 5]{FW} that the degeneration of $Y \times_{\overline{Y}} Y^+$ to the normal cone 
$$\Bl_{X_k \times \{0\}}(Y \times \A^1) \times_{\A^1} \Bl_{X_k^+ \times \{0\}}(Y^+ \times \A^1)$$ 
breaks up into the correspondence $Z(k,N)$ (section \ref{sec:somegeometry}) for the the local version of the stratified Mukai flop and into the correspondence $\Bl_{X_k} Y \times_{\Bl_{S_k} \overline{Y}} \Bl_{X_k^+} Y^+$ for a stratified Mukai flop of type $A_{N-2,k-1}$. Then using Theorem \ref{thm:main3} and induction one can imagine proving the equivalence on the special fibre (the fibre over $0 \in \A^1$). Since a kernel inducing an equivalence fibre-wise is an open condition this would then imply the equivalence on the general fibre too.  

\subsection{The stratified Mukai flop of type D}
We now briefly discuss the {\em stratified Mukai flop of type $D_{2m+1}$}\index{stratified Mukai flop of type D}.

\subsubsection{The local model}

Fix a symmetric, non-degenerate bilinear form $\la \cdot, \cdot \ra$ on $\C^{2N}$. Denote by $\IG(k,2N)$ the {\em isotropic Grassmannian}\index{isotropic Grassmannian} parametrizing isotropic $k$-planes in $\C^{2N}$. When $k=N$ it turns out $\IG(N,2N)$ has two components denoted $\IG(N,2N)^-$ and $\IG(N,2N)^+$. Two isotropic planes $V,V' \subset \C^{2N}$ belong to the same component if and only if $\dim(V \cap V') \equiv N \bmod 2$. 

The cotangent bundles of $\IG(N,2N)^\pm$ can be described as
$$T^\star \IG(N,2N)^\pm = \{(X,V) \in \so(2N) \times \IG(N,2N)^\pm: X(\C^{2N}) \subset V, X(V) \subset 0\}$$
where $X \in \so(2N)$ is a skew-symmetric matrix meaning that $\la Xv, w \ra = - \la v, Xw \ra$. Now consider the map 
$$ip_-: T^\star \IG(N,2N)^- \longrightarrow \IB(N,2N)$$ 
given by forgetting $V$ where $\IB(N,2N) := \{ X \in \so(2N) : X^2 = 0\}$. A general point $X \in \IB(N,2N)$ has $\dim (\ker X)$ equal to $N$ or $N-1$ depending on whether $N$ is even or odd. This essentially comes down to the fact that a skew-symmetric matrix of size $N$ has rank at most $N-1$ if $N$ is odd but can have full rank if $N$ is even. 

So there are two cases to consider. If $N$ is even then $\IB(N,2N)$ has two components. Two general points $X_1,X_2 \in \IB(N,2N)$ lie in the same component if and only if $X_1(\C^{2N}) \cap X_2(\C^{2N})$ is even. Then one component has a resolution given by $T^\star \IG(N,2N)^-$ and the other component has a resolution given by $T^\star \IG(N,2N)^+$. 

If $N$ is odd then $\IB(N,2N)$ only has one component. A resolution of this component is the variety 
\begin{equation}\label{eq:12}
\{(X,V): 0 \rightarrow V \xrightarrow{2} V^\perp \rightarrow \C^{2N}, X(\C^{2N}) \subset V, X(V^\perp) \subset 0\}
\end{equation}
where $X \in \so(2N)$ and $V \in \IG(N-1,2N)$. On the other hand, there is a natural map from 
\begin{equation}\label{eq:13}
\{(X,V,V'): 0 \xrightarrow{N-1} V \xrightarrow{1} V' \xrightarrow{1} V^\perp \xrightarrow{N-1} \C^{2N}, X(\C^{2N}) \subset V, X(V^\perp) \subset 0\}
\end{equation}
to (\ref{eq:12}). This map is everywhere $2:1$ since the fibres are all isomorphic to $\IG(1,2)$ using the restriction of $\la \cdot, \cdot \ra$ to $V^\perp/V$ and $\IG(1,2)$ is the disjoint union of two points. Forgetting $V$ and $V^\perp$ we get a generically one-to-one map from (\ref{eq:13}) to the two connected components in $T^\star \IG(N,2N)$. Thus if $N$ is odd we get the following diagram (in analogy with (\ref{eq:10}))
\begin{equation}\label{eq:11}
\xymatrix{
T^* \IG(N,2N)^- \ar[dr]_{\ip^-} & & T^* \IG(N,2N)^+ \ar[dl]^{\ip^+} \\
& \IB(N,2N) &
}
\end{equation}
This is the local model for the stratified Mukai flop of type $D_{2m+1}$ where $N=2m+1$. We equip everything with the $\C^\times$-action acting on the fibres of $T^\star \IG(N,2N)^\pm$ just like in the case of $T^\star \bG(k,N)$. 

\begin{remark}
The type A Grassmannian $\bG(k,N)$ corresponds to the minuscule $GL(N)$ representation $\Lambda^k(\C^N)$. In the case of $D_N$ there are three minuscule representations. One of them corresponds to $\IG(1,2N)$ while the other two correspond to $\IG(N,2N)^-$ and $\IG(N,2N)^+$. When $N$ is even the latter two respresentations are self dual but when $N$ is odd they are dual to each other. This is the representation theoretic manifestation of the dichotomy above. 
\end{remark}

{\bf Example.} Let us briefly examine $\IB(1,2)$ and $\IB(2,4)$. We fix the bilinear form $\left( \begin{matrix} 0 & I \\ I & 0 \end{matrix} \right)$ on $\C^{2N}$ and write a general element of $\so(2N)$ as $\left( \begin{matrix} A & B \\ C & D \end{matrix} \right)$ where $A,B,C,D$ are $N \times N$ matrices. The condition that it be skew-symmetric translates into 
$$A+D^t = 0 \text{ and } B+B^t = 0 = C+C^t$$
while the condition that it squares to zero is equivalent to 
$$A^2+BC=0, AB = BA^t \text{ and } CA = A^tC.$$
Now, if $N=1$ then the first condition implies that $B=0=C$ and $D=-A$ and the second condition says $A=0$ so that $\IB(1,2)$ consists of just a point. 

If $N=2$ then an elementary calculation (which we omit) shows that $\IB(2,4)$ has two possible types of solutions. The first is of the form $\left( \begin{matrix} A & 0 \\ 0 & -A \end{matrix} \right)$ where $A$ is a $2 \times 2$ matrix with $\det(A)=0=\tr(A)$ (i.e. a $2$-dimensional quadric cone). The second solution is of the form
$$\left( \begin{matrix} 
u & 0 & 0 & x \\
0 & -u & -x & 0 \\
0 & y & -u & 0 \\
-y & 0 & 0 & u 
\end{matrix} \right)$$
where $u^2=xy$. So $\IB(2,4)$ is the union of two $2$-dimensional quadric cones which intersect only at their apex.

\subsubsection{The general model}

Fix $N=2m+1$ from now on. The varieties $T^\star \IG(N,2N)^\pm$ have a natural filtration 
$$T^\star \IG(N,2N)^\pm \supset T^\star \IG(N,2N)^\pm_1 \supset \dots \supset T^\star \IG(N,2N)^\pm_m$$
where $T^\star \IG(N,2N)^\pm_i$ corresponds to the locus where $\dim(\ker X) \ge N+1+2i$. As before, we denote the image of $T^\star \IG(N,2N)^\pm_i$ by $S_i$. The subvariety $T^\star \IG(N,2N)^\pm_i \setminus T^\star \IG(N,2N)^\pm_{i+1}$ consists of the locus where $\dim(\ker X) = N+1+2i$ and hence is isomorphic to 
\begin{align*}
\{(X,V,V'): & 0 \xrightarrow{N-1-2i} V' \xrightarrow{2i+1} V \xrightarrow{2i+1} {V'}^\perp \xrightarrow{N-1-2i} \C^{2N} \\
&  X(\C^{2N}) = V', X({V'}^\perp) = 0\}
\end{align*}
since $V'$ can be recovered as $X(\C^{2N})$. Restricting $\la \cdot, \cdot \ra$ to ${V'}^\perp/V'$ we find that $V/V'$ is isotropic inside ${V'}^\perp/V'$. Thus the restriction of $\ip^\pm$ to $T^\star \IG(N,2N)^\pm_i \setminus T^\star \IG(N,2N)^\pm_{i+1}$ is a $\IG(2i+1, 4i+2)^\pm$-fibration onto its image $S_i \setminus S_{i-1}$. 

Motivated by this structure, one can define a stratified Mukai flop of type $D_{2m+1}$ just like in the type A case. In other words, one has subvarieties 
$$X_m^- \subset \dots \subset X_1^- \subset Y^- \text{ and } X_m^+ \subset \dots X_1^+ \subset Y^+$$
and maps $Y^- \xrightarrow{\ip^-} \overline{Y} \xleftarrow{\ip^+} Y^+$ satisfying the same conditions as before. The difference is that $W$ is now a rank $2N$ vector bundle equipped with a fibre-wise non-degenerate, symmetric bilinear form and $\ip^\pm|_{X_m} \rightarrow S_m$ is $\IG(N,W)^\pm \rightarrow S_m$ which is the relative isotropic Grassmannian.

\subsection{Equivalences in type D}

Once again we can consider the fibre product 
$$\IZ(N) :=  T^\star \IG(N,2N)^- \times_{\IB(N,2N)} T^\star \IG(N,2N)^+$$ 
but, as before, we cannot expect $\O_{\IZ(N)}$ to induce an equivalence. On the other hand, $\IZ(N)$ is made up of $m+1$ components $\IZ_0(N), \dots , \IZ_m(N)$ of dimension $N(N-1)$ where 
\begin{align}
\IZ_s(N) &= \{(X,V,V'): \dim(\ker X) \ge N+2s+1 \text{ and } \dim(V \cap V') \ge 2(m-s), \nonumber \\
& X(\C^{2N}) \subset V, X(\C^{2N}) \subset V', X(V) \subset 0, X(V') \subset 0\} \nonumber
\end{align}
inside $T^\star \IG(N,2N)^- \times  T^\star \IG(N,2N)^+$. One can define open subvarieties $\IZ^o_s(N) \subset \IZ_s(N)$ by imposing the additional condition 
$$\dim(\ker X) + \dim(V \cap V') \le 2N+2$$
just as we did to define $Z^o_s(k,N)$. It is not difficult to check that the open subvariety inside $\IZ^o_s(N) \cap \IZ^o_{s+1}(N)$ given by the condition 
$$\dim(\ker X) = N+2s+3 \text{ and } \dim(V \cap V') = 2(m-s)$$
is co-dimension one inside $\IZ_s(N)$ and $\IZ_{s+1}(N)$. This is completely analogous to the situation in type A. So one should strongly expect an analogue of Theorem \ref{thm:main3}.

\begin{conjecture}\label{conj:main1}
There exists a $\C^\times$-equivariant line bundles $\sIL(N)$ on $\IZ^o(N)$ such that $i_* j_* \sIL(N)$ induces an equivalence 
$$D(T^\star \IG(N,2N)^-) \xrightarrow{\sim} D(T^\star \IG(N,2N)^+).$$ 
Here $i$ and $j$ are the natural inclusions
$$\IZ^o(N) \xrightarrow{j} \IZ(N) \xrightarrow{i} T^\star \IG(N,2N)^- \times T^\star \IG(N,2N)^+.$$
\end{conjecture}

Since the cohomology of $\IG(k,2N)$ is fairly different than that of $\bG(k,N)$ it does not seem possible to construct a categorical $\sl_2$ action on cotangent bundles to isotropic Grassmannians. Nevertheless, one can imagine that some sort of action still exists. Finally, note that $T^\star \IG(N,2N)^\pm$ also have natural one-parameter deformations, defined just like in the type A case. This leads to a stratified Atiyah flop of type D. One can also conjecture and study derived equivalences in this case. 

\begin{remark} There are also stratified flops of type E which show up naturally in the birational geometry of resolutions of nilpotent orbit closures. See, for instance, \cite{CF} for a description of these. Most questions mentioned above in the case of type D flops also remain valid for type E. 
\end{remark}

\section{Further topics}\label{sec:topics}

{\bf Deformation quantization:}\index{deformation quantization} The category of $D$-modules on $\bG(k,N)$ can be deformed to the category of coherent sheaves on $T^\star \bG(k,N)$. The specialization map from $D$-modules to coherent sheaves is given by taking the associated graded. Now consider the open subset 
$$j: U \hookrightarrow \bG(k,N) \times \bG(N-k,N)$$
defined as the locus $(V,V')$ where $V \cap V' = 0$. It turns out that the push-forward $j_* \O_U$ of the $D$-module $\O_U$ is a $D$-module on $\bG(k,N) \times \bG(N-k,N)$ which induces an equivalence $\Dmod(\bG(k,N)) \xrightarrow{\sim} \Dmod(\bG(N-k,N))$. In \cite{CDK} we check that the associated kernel of this equivalence is actually the kernel $\sT(k,N) \in D(T^\star \bG(k,N) \times T^\star \bG(N-k,N))$.  

\begin{remark} Calculating the associated graded of a $D$-module is quite difficult in general. In \cite{CDK} we compute the associated graded by first constructing a categorical $\sl_2$ action on categories of $D$-modules on $\sqcup_k \bG(k,N)$ and then showing that it agrees, via the associated graded map, with the categorical $\sl_2$ action on coherent sheaves on $\sqcup_k T^\star \bG(k,N)$. So this approach does not really give an entirely different proof that $\sT(k,N)$ is invertible. 
\end{remark}

The category of $D$-modules on $\bG(k,N)$ is an example of a deformation quantization of the category of coherent sheaves on $T^\star \bG(k,N)$. But deformation quantizations also exists for category of coherent sheaves on quiver varieties, for instance. Understanding how these categorical Lie algebra actions and the corresponding equivalences deform to these deformation quantizations is a little explored but interesting problem. 

{\bf Flops as moduli spaces:} Bridgeland \cite{B} describes a way to construct the Atiyah flop $Y^+$ of a 3-fold $p: Y \rightarrow \overline{Y}$ as the moduli of perverse coherent sheaves on $Y$ (the definition of these perverse sheaves uses the map $p$). Then the universal family over the product $Y \times Y^+$ induces the derived equivalence $D(Y) \xrightarrow{\sim} D(Y^+)$. 

Can you generalize this result to other Atiyah or Mukai flops? This question seems difficult (but also interesting) in part because, as we saw in section \ref{sec:G(2,4)}, the autoequivalence of $D(\widetilde{T^\star \bG(2,4)})$ is not induced by the structure sheaf of the natural fibre product.


\begin{thebibliography}{E-G-S}

\bibitem[Ad]{Ad}
N. Addington, New derived symmetries of some Hyperk\"{a}hler varieties; \textsf{arXiv:1112.0487v1}. 

\bibitem[An]{An}
R. Anno, Spherical functors; \textsf{math.CT/0711.4409}.

\bibitem[B]{B}
T. Bridgeland, Flops and derived categories, \textit{Invent. Math.} {\textbf{147}} (2002) no. 3, 613–-632; \textsf{math.AG/0009053}. 

\bibitem[BR]{BR}
R. Bezrukavnikov and S. Riche, Affine braid group actions on derived categories of Springer resolutions; \textsf{arXiv:1101.3702}.

\bibitem[C1]{C1}
S. Cautis, Equivalences and stratified flops, \emph{Compositio Math.} \textbf{148} (2012), no. 1, 185--209; \textsf{math.AG/0909.0817}.

\bibitem[C2]{C2}
S. Cautis, Clasp technology to knot homology via the affine Grassmannian; \textsf{arXiv:1207.2074}. 

\bibitem[CDK]{CDK}
S. Cautis, C. Dodd and J. Kamnitzer, Categorical actions on quiver varieties: from $\mathcal{D}$-modules to coherent sheaves; (in preparation).

\bibitem[CK1]{CK1}
S. Cautis and J. Kamnitzer, Knot homology via derived categories of coherent sheaves I, sl(2) case, \textit{Duke Math. J.} \textbf{142} (2008), no. 3, 511--588. \textsf{math.AG/0701194}.

\bibitem[CK2]{CK2}
S. Cautis and J. Kamnitzer, Knot homology via derived categories of coherent sheaves II, sl(m) case, \textit{Invent. Math.} \textbf{174} (2008), no. 1, 165--232. \textsf{math.AG/0710.3216}.

\bibitem[CK3]{CK3}
S. Cautis and J. Kamnitzer, Braid groups and geometric categorical Lie algebra actions, \emph{Compositio Math.} \textbf{148} (2012), no. 2, 464--506; \textsf{arXiv:1001.0619}.

\bibitem[CKL1]{CKL1} S. Cautis, J. Kamnitzer and A. Licata, Categorical geometric skew Howe duality, \textit{Invent. Math.} \textbf{180} (2010), no. 1, 111--159. \textsf{math.AG/0902.1795}.

\bibitem[CKL2]{CKL2} S. Cautis, J. Kamnitzer and A. Licata, 
Coherent sheaves and categorical $\sl_2$ actions, \emph{Duke Math. J.} \textbf{154} (2010), no. 1, 135--179; \textsf{math.AG/0902.1796}.

\bibitem[CKL3]{CKL3} S. Cautis, J. Kamnitzer and A. Licata, Derived equivalences for cotangent bundles of Grassmannians via categorical $sl_2$ actions, \emph{J. Reine Angew. Math.} (to appear); \textsf{math.AG/0902.1797}. 

\bibitem[CKL4]{CKL4}
S. Cautis, J. Kamnitzer and A. Licata, Coherent sheaves on quiver varieties and categorification; \textsf{arXiv:1104.0352}.

\bibitem[CF]{CF}
P. E. Chaput and B. Fu, On stratified Mukai flops, \textit{Math. Res. Lett.} \textbf{14} (2007), no. 6, 1055--1067.

\bibitem[CG]{CG}
N. Chriss and V. Ginzburg, \textit{Representation theory and complex geometry}, Birk\"auser, 1997.

\bibitem[CR]{CR}
J. Chuang and R. Rouquier, Derived equivalences for symmetric groups and $\sl_2$-categorification, \textit{Ann. of Math.} \textbf{167} (2008), no. 1, 245--298; \textsf{math.RT/0407205}.

\bibitem[F]{F}
B. Fu, Extremal contractions, stratified Mukai flops and Springer maps, \textit{Adv. Math.} \textbf{213} (2007), 165--182; \textsf{math.AG/0605431}.

\bibitem[FW]{FW}
B. Fu and C.-L. Wang, Motivic and quantum invariance under stratified Mukai flops, \textit{J. Differential Geometry}, \textbf{80} (2008), 261--280.

\bibitem[Ho]{Ho} 
R.P. Horja, Derived Category Automorphisms from Mirror Symmetry, \textit{Duke Math. J.} \textbf{127} (2005), 1--34; \textsf{math.AG/0103231}. 

\bibitem[HT]{HT}
D. Huybrechts and R. Thomas, {$\P$-objects and autoequivalences of derived categories}, \textit{Math. Res. Lett.} \textbf{13} (2006), no. 1, 87--98; \textsf{math.AG/0507040}.

\bibitem[K1]{K1}
Y. Kawamata, D-equivalence and K-equivalence, \textit{J. Diff. Geom.}, \textbf{61} (2002), 147-171; \textsf{math.AG/0205287}.

\bibitem[K2]{K2}
Y. Kawamata, Derived equivalence for stratified Mukai flop on $\bG(2,4)$, \textit{Mirror symmetry. V}, AMS/IP Stud. Adv. Math., \textbf{38}, Amer. Math. Soc. (2006), 285--294; \textsf{math.AG/0503101}.

\bibitem[KL1]{KL1}
M. Khovanov and A. Lauda, A diagrammatic approach to categorification of quantum groups I, \textit{Represent. Theory} \textbf{13} (2009), 309--347; \textsf{math.QA/0803.4121}.

\bibitem[KL2]{KL2}
M. Khovanov and A. Lauda, A diagrammatic approach to categorification of quantum groups II, \textit{Trans. Amer. Math. Soc.} \textbf{363} (2011), 2685--2700; \textsf{math.QA/0804.2080}.

\bibitem[KL3]{KL3}
M. Khovanov and A. Lauda, A diagrammatic approach to categorification of quantum groups III, \textit{Quantum Topology} \textbf{1}, Issue 1 (2010), 1--92; \textsf{math.QA/0807.3250}.

\bibitem[KT]{KT}
M. Khovanov and R. Thomas, Braid cobordisms, triangulated categories, and flag varieties, \textit{Homology, Homotopy, Appl.} \textbf{9} (2007), 19--94; \textsf{math.QA/0609335}.

\bibitem[L]{L}
A. Lauda, A categorification of quantum $\sl_2$, \textit{Adv. Math.}, \textbf{225}, no. 6, (2010), 3327--3424; \textsf{arXiv:0803.3652v2}.

\bibitem[M]{M}
E. Markman, Brill-Noether duality for moduli spaces of sheaves on K3 surfaces, \textit{J. Algebraic Geom.} \textbf{10} (2001), 623--694; \textsf{math.AG/9901072v1}.

\bibitem[Nak1]{Nak1} H. Nakajima, Quiver varieties and Kac-Moody algebras. \textit{Duke Math. J.} 91 (1998), no. 3, 515--560.

\bibitem[Nak2]{Nak2}
H. Nakajima, Quiver varieties and finite-dimensional representations of quantum affine algebras. \textit{J. Amer. Math. Soc.} \textbf{14} (2001), no.1, 145--238.

\bibitem[Nam1]{Nam1}
Y. Namikawa, Mukai flops and derived categories, \textit{J. Reine Angew. Math.} \textbf{560} (2003), 65--76; \textsf{math.AG/0203287}.

\bibitem[Nam2]{Nam2}
Y. Namikawa, Mukai flops and derived categories II, \textit{Algebraic structures and moduli spaces}, CRM Proc. Lecture Notes, \textbf{38}, Amer. Math. Soc. (2004), 149--175; \textsf{math.AG/0305086}.

\bibitem[Nam3]{Nam3}
Y. Namikawa, Birational geometry of symplectic resolutions of nilpotent orbits, \textit{Advanced Studies in Pure Mathematics} 45, (2006), Moduli spaces and Arithmetic geometry (Kyoto, 2004), 75--116.

\bibitem[Ric]{Ric}
S. Riche, Geometric braid group action on derived category of coherent sheaves,  \textit{Represent. Theory} \textbf{12} (2008), 131--169.

\bibitem[Ro1]{Rold}
R. Rouquier, Categorification of $\sl_2$ and braid groups, \textit{Trends in representation theory of algebras and related topics} (2006) Amer. Math. Soc., 137--167. 

\bibitem[Ro2]{Rnew}
R. Rouquier, 2-Kac-Moody algebras; \textsf{math.RT/0812.5023}.

\bibitem[Roz]{Roz}
L. Rozansky, An infinite torus braid yeilds a categorified Jones-Wenzl projector; \textsf{arXiv:1005.3266v1}.

\bibitem[ST]{ST}
P. Seidel and R. Thomas, Braid group actions on derived categories of coherent sheaves, \textit{Duke Math J.}, \textbf{108}, (2001), 37--108; \textsf{math.AG/0001043}.

\end{thebibliography}
\end{document}